\documentclass{amsart}

\usepackage{amsthm}
\usepackage{amsmath}
\usepackage{amssymb}
\usepackage{amscd}
\usepackage{mathrsfs}
\usepackage{verbatim}
\usepackage{hyperref}

\input{colordvi}

\allowdisplaybreaks

\renewcommand{\a}{\alpha}

\newcommand{\eps}{\varepsilon}
\renewcommand{\l}{\lambda}
\newcommand{\g}{\gamma}
\renewcommand{\l}{\lambda}
\renewcommand{\O}{\Omega}

\newcommand{\DEQS}{\begin{eqnarray*}}
\newcommand{\EEQS}{\end{eqnarray*}}

\newcommand\del[1]{}
\newcommand\think[1]{}
\newcommand\new[1]{}
\newcommand\zus[1]{}

\def\R{{\mathbb R}}
\def\N{{\mathbb N}}
\def\C{{\mathbb C}}
\def\E{{\mathbb E}}
\def\P{{\mathbb P}}

\newcommand{\s}{^{\ast}}
\newcommand{\n}{\Vert}

\newcommand{\calF}{\mathcal F}
\newcommand{\calH}{\mathcal H}
\newcommand{\calL}{\mathcal L}
\newcommand{\calM}{\mathcal M}

\newcommand{\MA}{{\bf (A)}}
\newcommand{\MF}{{\bf (F)}}
\newcommand{\MG}{{\bf (G)}}

\newcommand{\Vpc}[1]{V^{#1,p}_{\rm c}}

\theoremstyle{plain}
\newtheorem{theorem}{Theorem}[section]
\theoremstyle{remark}
\newtheorem{remark}[theorem]{Remark}

\theoremstyle{plain}
\newtheorem{corollary}[theorem]{Corollary}
\newtheorem{lemma}[theorem]{Lemma}
\newtheorem{proposition}[theorem]{Proposition}
\newtheorem{definition}[theorem]{Definition}

\DeclareMathOperator*{\esssup}{ess\,sup}
\newcommand{\inv}[1]{\frac{1}{#1}}
\newcommand{\tinv}[1]{\tfrac{1}{#1}}

\newcommand{\maxsym}{\vee}
\newcommand{\minsym}{\wedge}

\newcounter{gr1}

\begin{document}
\title[Perturbations of SDEs in UMD Banach spaces.]
{A perturbation result for quasi-linear stochastic differential equations in UMD Banach spaces.}

\author{Sonja Cox}
\address{Delft Institute of Applied Mathematics\\
Delft University of Technology \\ P.O. Box 5031\\ 2600 GA Delft\\The Netherlands}
\email{S.G.Cox@tudelft.nl}

\author{Erika Hausenblas}
\address{Pepartmnent of Mathematics and Informationtechnology\\Montanuniversity Leoben\\Fr.\ Josefstr. 18\\8700 Leoben\\Austria}
\email{erika.hausenblas@unileoben.ac.at}

\begin{abstract}
We consider the effect of perturbations of $A$ on the solution to the following quasi-linear parabolic
stochastic differential equation set in a \textsc{umd} Banach space $X$:
\begin{equation}\label{SDE_abstract}\tag{SDE}
\left\{ \begin{aligned} dU(t) & = AU(t)\,dt +F(t,U(t))\,dt + G(t,U(t))\,dW_H(t), \quad t>0;\\
U(0)&=x_0. \end{aligned}\right.
\end{equation}
Here $A$ is the generator of an analytic $C_0$-semigroup on $X$, $G:[0,T]\times X\rightarrow \calL(H,X_{\theta_G}^{A})$ and $F:[0,T]\times
X\rightarrow X_{\theta_F}^{A}$ for some $\theta_G>-\inv{2}$, $\theta_F>  -\frac{3}{2}+\inv{\tau}$, where $\tau$ is the type of $X$. We assume $F$ and $G$ to satisfy certain global Lipschitz and linear growth conditions.\par
Let $A_0$ denote the perturbed operator and $U_0$ the solution to \eqref{SDE_abstract} with $A$ substituted by $A_0$. We provide estimates for $\n U - U_0 \n_{L^p(\Omega;C([0,T];X))}$ in terms of $D_{\delta}(A,A_0):=\n R(\lambda:A)-R(\lambda:A_0)\n_{\calL(X^{A}_{\delta-1},X)}$. Here $\delta\in [0,1]$ is assumed to satisfy $0\leq \delta < \min\{\tfrac{3}{2}-\tinv{\tau}+\theta_F,\, \tinv{2}-\tinv{p}+\theta_G \}$.\par
The work is inspired by the desire to prove convergence of space approximations of \eqref{SDE_abstract}. In this article we prove convergence rates for the case that $A$ is approximated by its Yosida approximation.
\vspace{0.3cm}
\end{abstract}

\maketitle

\section{Introduction}\label{mnb}

In this article we consider the effect of perturbations of $A$ on the solution to the following stochastic differential equation set in a \textsc{umd} Banach space $X$:
\begin{equation}\label{SDE1}\tag{SDE}
\left\{ \begin{aligned} dU(t) & = AU(t)\,dt +F(t,U(t))\,dt + G(t,U(t))\,dW_H(t), \quad t>0;\\
U(0)&=x_0. \end{aligned}\right.
\end{equation}
Here $A$ is the generator of an analytic $C_0$-semigroup on $X$, $G:[0,T]\times
X\rightarrow \calL(H,X_{\theta_G}^{A})$ and $F:[0,T]\times X\rightarrow
X_{\theta_F}^{A}$ for some $\theta_G>-\inv{2}$, $\theta_F>
-\frac{3}{2}+\inv{\tau}$, where $\tau\in [1,2)$ is the type of $X$. We use $X_{\delta}^A$ to denote the fractional domain or extra\-po\-lation space corresponding to $A$. We assume $F$ and $G$ to satisfy
certain global Lipschitz and linear growth conditions, see Section \ref{ss:setting} below. The framework in which we consider \eqref{SDE1} is precisely the one for which existence and uniqueness of a solution has been proven in \cite{NVW08}. A typical example of a stochastic partial differential equation that fits into this framework is a one-dimensional parabolic stochastic partial differential equation driven by white noise.\par
The main motivation to study the effect of perturbations of $A$ on solutions to \eqref{SDE} is the desire to prove convergence of certain numerical schemes for approximations in the space dimension. In fact, in \cite{CoxHau:11b} we demonstrate how the perturbation result proven in this article can be used to obtain pathwise convergence of certain Galerkin and finite element methods for \eqref{SDE} in the case that $X$ is Hilbertian. Here we focus more on the theoretical aspects, and demonstrate how our perturbation result can be used to prove convergence of the solution processes if $A$ is replaced by its Yosida approximation.\par
With applications to numerical approximations in mind, we assume the perturbed equation to be set in a (possibly finite dimensional) closed subspace $X_0$ of $X$. We assume that there exists a bounded projection $P_0:X\rightarrow X_0$ such that $P_0(X)=X_0$. Let $i_{X_0}$ be the canonical embedding of $X_0$ in $X$ and let $A_0$ be a generator of an analytic $C_0$-semigroup $S_0$ on $X_0$. In the setting of numerical approximations, $A_0$ would be a suitable restriction of $A$ to the finite dimensional space $X_0$.\par
The perturbed equation we consider is the following stochastic differential equation:
\begin{equation}\label{SDE0_intro}\tag{SDE$_0$}
\left\{ \begin{aligned} dU^{(0)}(t) & =A_0U^{(0)}(t)\,dt +P_0F(t,U^{(0)}(t))\,dt \\
&\quad + P_0G(t,U^{(0)}(t))\,dW_H(t), \qquad\qquad  t>0;\\
U^{(0)}(0)& = P_0x_0. \end{aligned}\right.
\end{equation}
Our main result, Theorem \ref{t:app} below, states that if we have:
\begin{align}\label{defDAAintro}
D_{\delta}(A,A_0):=  \n R(\lambda_0:A)-i_{X_0} R(\lambda_0:A_0)P_0 \n_{\calL(X_{\delta-1}^{A},X)} < \infty,
\end{align}
for some $\delta\geq 0$ satisfying
$$0\leq \delta < \min\{\tfrac{3}{2}-\tinv{\tau}+\theta_F,\, \tinv{2}+\theta_G \},$$
and $x_0\in L^p(\Omega;\calF_0;X_{\delta}^{A})$ for $p\in (2,\infty)$ such that $\inv{p}\leq \inv{2}+\theta_G-\delta$, then there exists a solution to $\eqref{SDE0_intro}$ in $L^p(\Omega;C([0,T];X_0))$ and moreover:
\begin{equation*}
\begin{aligned}
\n U - i_{X_0}U^{(0)} \n_{L^p(\Omega;C([0,T];X))}&
\lesssim D_{\delta}(A,A_0)(1+\n x_0
\n_{L^p(\Omega;X_{\delta}^{A})}).
\end{aligned}
\end{equation*}
Note that if $\delta<1$ then a priori it is not obvious whether $D_{\delta}(A,A_0)$ is finite.\par
As a corollary of Theorem \ref{t:app} we obtain an estimate in the H\"older norm provided we compensate for the initial values (see Corollary \ref{cor:Holder} below), i.e., for $\lambda\in [0,\inv{2})$ satisfying
$$0\leq \lambda < \min\{\tfrac{3}{2}-\tinv{\tau}-(\delta-\theta_F)\maxsym 0, \tinv{2}-\tinv{p}-(\delta-\theta_G)\maxsym 0\},$$
we have:
\begin{align*}
\n U - Sx_0 - i_{X_0}(U^{(0)} - S_0P_0x_0) \n_{L^p(\Omega;C^{\lambda}([0,T];X))}&
\lesssim D_{\delta}(A,A_0)(1+\n x_0
\n_{L^p(\Omega;X_{\delta}^{A})}).
\end{align*}\par
Our results imply that if $(A_n)_{n\in\N}$ is a family of generators of analytic semigroups such that the resolvent of $A_n$ converges to the resolvent of $A$ in $\calL(X_{\delta-1}^A,X)$ for some $\delta\in [0,1]$ (and $(A_n)_{n\in\N}$ is uniformly analytic), then the corresponding solution processes $U_n$ converge to the actual solution in $L^p(\Omega;C([0,T];X))$ and the convergence rate is given by $D_{\delta}(A,A_n)$.\par
In particular, we may apply Theorem \ref{t:app} to the Yosida approximation of $A$. In this case we assume $\theta_F$ and $\theta_G$ are positive. The $n^{th}$ Yosida approximation of $A$ is given by $A_n=nAR(n:A)$, and we let $U^{(n)}$ denote the solution to \eqref{SDE1} where $A$ is substituted by $A_n$. By applying Theorem \ref{t:app} we obtain that for $\eta\in [0,1]$ and $p\in (2,\infty)$ such that
$$\eta<\min\{\tfrac{3}{2}-\tinv{\tau}+\theta_F,\tfrac{1}{2}-\tinv{p}+\theta_G \}$$ we have, assuming $x_0\in L^p(\Omega,\mathcal{F}_0;X_{\eta}^A)$:
\begin{align*}
\n U - U^{(n)}\n_{L^p(\Omega;C([0,T];X))}
&\lesssim  n^{-\eta}(1+\n x_0 \n_{L^p(\Omega;X_{\eta}^{A})}).
\end{align*}
See also Theorem \ref{t:yosida}.\par
It was proven in \cite{KunNeer11} that if $(A_n)_{n\in\N}$ is a family of generators of analytic semigroups such that the resolvent of $A_n$ converges to the resolvent of $A$ in the strong operator topology, then the corresponding solution processes $U_n$ converge to the actual solution in $L^p(\Omega;C([0,T];X))$. However, the approach taken in that article does not provide convergence rates and requires $\theta_F,\theta_G\geq 0$.\par

Another article in which approximations of solutions to \eqref{SDE} are considered in the context of 
perturbations on $A$ is \cite{Brze:97}. In that article, it is assumed that $X$ is a \textsc{umd}
space with martingale type 2. In Section 5 of that article the author considers approximations of
$A$, $F$, $G$ and of the noise. Translated to our setting, the author assumes the perturbed
operator $A_0$ to satisfy $X^{A_0}_{\theta_F}=X^{A}_{\theta_F}$ and
$X^{A_0}_{\theta_G}=X^{A}_{\theta_G}$ (in particular, $X_0$ cannot be finite-dimensional).\par

A natural question to ask is how the type of perturbation studied here relates to the perturbations known in the literature. In  \cite{DeschSchap:87}, \cite{Jung:00}, and \cite{Rob:77} (see also \cite[Chapter III.3]{EngNa:00}) one has derived conditions for perturbations of $A$ that lead to an estimate of the type $\n S(t) - S_0(t)\n_{\calL(X)}=\mathcal{O}(t)$. In light of Proposition \ref{p:gbddSG} below these results are comparable to our results if we were to take $\alpha=-1$. In particular, \cite[Theorem III.3.9]{EngNa:00} gives precisely the same results as Proposition \ref{p:gbddSG}, but then for the case $\alpha=-1$ and $\beta=0$.\par

The proof of our perturbation result (Theorem \ref{t:app}) requires regularity results for stochastic convolutions. As the convolution under consideration concerns the difference between two semigroups instead of a single semigroup, the celebrated factorization method of \cite{DKZ} fails. Therefore we prove a new result on the regularity of stochastic convolutions, see Lemma \ref{lem:stochConv} below. This lemma in combination with some randomized boundedness results on $S-S_0P_0$ form the key ingredients of the proof Theorem \ref{t:app}. \par
The set-up of this article is as follows: Section \ref{sec:prelim} contains the preliminaries; i.e., the necessary results on analytic semigroups, vector-valued stochastic integration theory, and $\gamma$-boundedness. In that section we also state the precise assumptions on $A$, $F$, and $G$ in \eqref{SDE}, and prove the regularity results for (stochastic) convolutions that we need in the proof of Theorem \ref{t:app}. In Section \ref{sec:app} we prove Theorem \ref{t:app} and in Section \ref{sec-yosida} we prove convergence for the Yosida approximations.\par
\subsection*{Notation}
For an operator $A$ on a Banach space $X$ we denote the resolvent set of $A$ by
$\rho(A)$, i.e., $\rho(A)\subset \mathbb{C}$ is the set of all the complex numbers
$\lambda\in \mathbb{C}$ for which $\lambda I-A$ is (boundedly) invertible. For
$\lambda\in\rho(A)$ we denote the resolvent of
$A$ in $\lambda$ by $R(\lambda:A)$, i.e., $R(\lambda:A)=(\lambda I -A)^{-1}$. The spectrum of
$A$, i.e., the complement of $\rho(A)$ in $\C$, is denoted by $\sigma(A)$.\par
For $X, Y$ Banach spaces we let $\calL(X,Y)$ be the Banach space of all bounded linear operators from $X$ to
$Y$ endowed
with the operator norm. For brevity we set $\calL(X):=\calL(X,X)$.\par
For $T>0$ and $\theta>0$ we take the following definition for the H\"older norm:
\begin{align*}
\n f \n_{C^{\theta} ([0,T];Y)} &:= \n f(0)\n_{Y} + \sup_{0\leq s<t\leq T} \frac{\n f(t)-f(s)\n_{Y}}{(t-s)^{\theta}}.
\end{align*}\par
We write $A \lesssim B$ to express that there exists a constant $C>0$ such that $A\leq C B$, and we write $A\eqsim B$ if
$A\lesssim B$ and $B\lesssim A$. Finally, for $X$ and $Y$ Banach spaces we write $X\simeq Y$ if $X$ and
$Y$ are isomorphic as Banach spaces.\par
\section{Preliminaries}\label{sec:prelim}
Throughout this section $X$, $Y$, and $Y_{i}$, $i\in \{1,2\}$, will be used to denote Banach spaces and $H$ will denote a Hilbert space.
\subsection{Analytic semigroups}
For $\delta \in [0,\pi]$ we define $$\Sigma_{\delta}:= \{ z\in \C\setminus\{0\} \,:\, |\arg(z)| < \delta\}.$$
We recall the definition of analytic $C_0$-semigroups \cite[Chapter 2.5]{Pazy:83}:
\begin{definition}
Let $\delta\in (0,\frac{\pi}{2})$. A $C_0$-semigroup $(S(t))_{t\geq 0}$ on $X$ is called
\emph{analytic
in} $\Sigma_{\delta}$ if
\begin{enumerate}
\item $S$ extends to an analytic function $S:\Sigma_{\delta}\rightarrow \calL(X)$;
\item $S(z_1+z_2)=S(z_1)S(z_2)$ for $z_1,z_2\in \Sigma_{\delta}$;
\item $\lim_{z\rightarrow 0; z\in \Sigma_\delta}S(z)x=x$ for all $x\in X$.
\end{enumerate}
\end{definition}
Typical examples of operators generating analytic $C_0$-semigroups are second-order elliptic operators. The theorem
below is obtained from \cite[Theorem 2.5.2]{Pazy:83} by straightforward adaptations and gives some characterizations of
analytic $C_0$-semigroups that we shall need.
\begin{theorem}\label{t:analchar}
Let $A$ be the generator of a $C_0$-semigroup $(S(t))_{t\geq0}$ on $X$. Let $\omega\in \R$ be such that
$(e^{-\omega t}S(t))_{t\geq 0}$ is exponentially stable. The following statements are equivalent:
\begin{enumerate}
\item\label{analdef} $S$ is an analytic $C_0$-semigroup on $\Sigma_{\delta}$ for some $\delta \in
(0,\frac{\pi}{2})$ and for every $\delta'<\delta$ there exists a constant $C_{1,\delta'}$ such that
$\n e^{-\omega z} S(z)\n\leq C_{1,\delta'}$ for all $z\in \Sigma_{\delta'}$.
\item\label{analres} There exists a $\theta \in (0,\frac{\pi}{2})$ such that $ \omega+\Sigma_{\frac{\pi}{2}+\theta}
\subset \rho(A)$,
and for every $\theta'\in (0,\theta)$ there exists a constant $C_{2,\theta'}>0$ such that:
\begin{align*}
|\lambda-\omega| \n R(\lambda:A) \n \leq C_{2,\theta'}, \quad \textrm{for all } \lambda\in
\omega+\Sigma_{\frac{\pi}{2}+\theta'}.
\end{align*}
\item\label{analdiff} $S$ is differentiable for $t>0$ (in the uniform operator topology), $\frac{d}{dt}S=AS$, and there
exists a constant $C_3$ such that:
\begin{align*}
t\n A S(t)\n \leq C_3 e^{\omega t}, \quad \textrm{for all } t>0.
\end{align*}
\end{enumerate}
\end{theorem}\par

This theorem justifies the following definition:
\begin{definition}\label{d:uniftype}
Let $A$ be the generator of an analytic $C_0$-semigroup on $X$. We say that $A$ is
\emph{of type }$(\omega,\theta,K)$, where $\omega\in\R$, $\theta\in (0,\frac{\pi}{2})$ and $K>0$, if
$\omega+\Sigma_{\frac{\pi}{2}+\theta} \subseteq \rho(A)$, $(e^{\omega t}S(t))_{t\geq 0}$ is exponentially stable, and $$|\lambda-\omega|\n R(\lambda:A) \n_{\calL(X)} \leq K
\quad \textrm{for all } \lambda \in \omega+\Sigma_{\frac{\pi}{2}+\theta}.$$
\end{definition}
\begin{remark}\label{r:analcont}
It follows from the aforementioned proof in \cite{Pazy:83} that the constants $\delta,  C_{1,\delta'}$;
$\delta'\in (0,\delta)$, $C_{2,\theta'}$; $\theta'\in (0,\theta)$, and $C_{3}$ in Theorem \ref{t:analchar} can be
expressed explicitly in terms of $\omega$, $\theta$, and $K$; for example, we may take $C_3= \tfrac{K}{\pi \cos
\theta}$.
\end{remark}

Note that if $A$ is the generator of an analytic $C_0$-semigroup of type $(\omega,\theta,K)$ then for all $\lambda \in
\omega(1+2(\cos\theta)^{-1})+\Sigma_{\frac{\pi}{2}+\theta}$ one has (noting that the choice of $\lambda$ implies
$|\lambda|>2|\omega|$ and hence $|\lambda-\omega|>||\lambda|-|\omega||\geq \inv{2}|\lambda|$):
\begin{align}\label{AResEst}
\n A R(\lambda:A)\n_{\calL(X)} = \n \lambda R(\lambda:A) - I \n \leq 1 + 2K.
\end{align}\par
Let $A$ be the generator of an analytic semigroup of type $(\omega,\theta,K)$ on $X$. We define the
extrapolation spaces of $A$ conform \cite[Section 2.6]{Pazy:83}; i.e., for $\delta>0$ and $\lambda \in \C$ such that
$\mathscr{R}e(\lambda)>\omega$ we define $X^{A}_{-\delta}$ to be the closure of $X$ under the norm $\n
x\n_{X^{A}_{-\delta}} := \n (\lambda I -A)^{-\delta} x \n_X$. We also define the fractional domain spaces of $A$, i.e.,
for $\delta>0$ we define $X^{A}_{\delta}=D((\lambda I -A)^{\delta})$ and $\n x\n_{X^{A}_{\delta}} := \n (\lambda I
-A)^{\delta} x \n_X$. One may check that regardless of the choice of $\lambda$ the extrapolation spaces and the
fractional domain spaces are uniquely determined up to isomorphisms: for $\delta>0$ one has $(\lambda I -A)^{\delta}(\mu
I -A)^{-\delta}\in \calL(X)$ and:
$$\n (\lambda I -A)^{\delta}(\mu I -A)^{-\delta}\n_{\calL(X)}\leq C(\omega,\theta,K,\lambda,\mu),$$
where $C(\omega,\theta,K,\lambda,\mu)$ denotes a constant depending only on $\omega,\theta,K,\lambda,$ and $\mu$.
Moreover, for $\delta,\, \beta\in \R$ one has $(\lambda I-A)^{\delta}(\lambda I-A)^{\beta}=(\lambda I-A)^{\delta+\beta}$
on $X^A_{\gamma}$, where $\gamma= \max\{\beta,\delta+\beta\}$ (see \cite[Theorem 2.6.8]{Pazy:83}).\par
Statement \eqref{analdiff} in Theorem \ref{t:analchar} can be extended; from the proof of \cite[Theorem 2.6.13]{Pazy:83}
we obtain that for an analytic $C_0$-semigroup $S$ of type $(\omega,\theta,K)$ generated by $A$ one has, for $\delta>0$:
\begin{align}\label{analyticSG-Est}
\n S(t)\n_{\calL(X,X^A_{\delta})} & \leq 2\big(\tfrac{K}{\pi \cos \theta}\big)^{\lceil \delta \rceil} t^{-\delta} e^{\omega t}.
\end{align}\par

The following interpolation result holds for the fractional domain spaces (see \cite[Theorem 2.6.10]{Pazy:83}):
\begin{theorem}\label{t:fracPowInt}
Let $A$ be the generator of an analytic $C_0$-semigroup on $X$ of type
$(\omega,\theta,K)$. Let $\delta\in (0,1)$ and $\lambda\in \C$ such that $\mathscr{R}e(\lambda)>\omega$. Then for every
$x\in D(A)$ we have:
\begin{align*}
\n (\lambda I - A)^{\delta} x \n& \leq 2(1+K)\n x\n^{1-\delta}\n (\lambda I - A) x \n^{\delta}.
\end{align*}
\end{theorem}
For more properties of $X_{\delta}^{A}$, $\delta\in \R$, we refer to \cite[Section 2.6]{Pazy:83}. \par
\subsection{Stochastic differential equations}
Throughout this section $X$ and $Y$ denote \textsc{umd} Banach spaces, and $(\Omega,(\mathcal{F}_t)_{t\geq 0},\P)$ denotes a probability space.\par
\subsubsection{Stochastic integration in \textsc{umd} Banach spaces}
Before turning to stochastic differential equations, we recall the basics concerning stochastic integration in
\textsc{umd} Banach spaces as presented in \cite{NVW07a}.
Recall that the \textsc{umd} property is a geometric Banach space property that is satisfied by all Hilbert spaces and
the `classical' reflexive function spaces, e.g.\ the $L^p$-spaces and Sobolev spaces $W^{k,p}$ for $k\in \N$ and $p\in
(1,\infty)$. For the precise definition of the \textsc{umd} property and for a more elaborate treatment of spaces
satisfying this property we refer to \cite{Bu3}.\par
Fix $T>0$. An {\em $H$-cylindrical Brownian motion over $(\Omega,(\mathcal{F}_t)_{t\geq
0},\P)$} is a mapping $W_{H}:L^2(0,T;H)\rightarrow L^2(\Omega)$ with the following properties:
\begin{enumerate}
\item for all $h\in L^2(0,T;H)$ the random variable $W_{H}(h)$ is Gaussian;
\item for all $h_1,h_2\in L^2(0,T;H)$ we have $\E W_{H}(h_1)W_{H}(h_2)=\langle h_1,h_2\rangle$;
\item for all $h\in H$ and all $t\in [0,T]$ we have that $W_H(1_{[0,t]}\otimes h)$ is $\mathcal{F}_t$-measurable;
\item for all $h\in H$ and all $s,t\in [0,T]$, $s\leq t$ we have that $W_H(1_{[s,t]}\otimes h)$ is independent of
$\mathcal{F}_s$.
\end{enumerate}
Formally, an $H$-cylindrical Brownian motion can be thought of as a standard Brownian motion taking values in the
Hilbert space $H$.\par

For the precise definition of the stochastic integral of a process $\Phi:[0,T]\times\Omega\rightarrow \calL(H,X)$ with
respect to $W_H$ we refer to \cite{NVW07a}. For our purposes it suffices to cite the characterization of such
stochastically integrable processes in terms of the so-called $\gamma$-radonifying norm of the process.\par
Let $\calH$ be a Hilbert space (we will take $\calH=L^2(0,T;H)$ later on). The Banach space $\gamma(\calH,X)$ is defined as the completion of $\calH\otimes X$ with
respect to the norm
$$ \Big\n \sum_{n=1}^N h_n\otimes x_n \Big\n_{\gamma(\calH,X)}^2 := \E \Big\n
\sum_{n=1}^N \gamma_n\otimes x_n \Big\n^2.$$
Here we assume that $(h_n)_{n=1}^N$ is an orthonormal sequence in $\calH$, 
$(x_n)_{n=1}^N$ is a sequence in $X$, and  $(\gamma_n)_{n=1}^N$ is a standard
Gaussian sequence on some
probability space. The space $\gamma(H,X)$ embeds continuously into $\calL(\calH,X)$
and it elements
are referred to as the {\em $\g$-radonifying} operators from $\calH$ to $X$. For properties of this norm and further details we refer to the survey paper \cite{Nee-survey}.\par

Let $(R,\mathcal{R},\mu)$ be a measure space. In that case $\gamma(R,H;X)$ is used as short-hand notation for
$\gamma(L^2(R;H),X)$; in particular, $\gamma(0,T,H;X)$ is short-hand notation for $\gamma(L^2(0,T;H);X)$. We use $\gamma(0,T;X)$ to denote $\gamma(L^2(0,T),X)$. If $X$ is a Hilbert space, and $(R,\mathcal{R},\mu)$ is a $\sigma$-finite measure space, then $\gamma(R,H;X)= L^2(R,\mathcal{L}_2(H,X))$ where $\mathcal{L}_2(H,X)$ denotes the
space of Hilbert-Schmidt operators from $H$ to $X$.\par

A process $\Phi:[0,\infty)\times \Omega\rightarrow \calL(H,X)$ is called
$H$-\emph{strongly measurable} if for every $h\in H$ the process $\Phi h$ is strongly measurable. The process is called
\emph{adapted} if $\Phi h$ is adapted 
for each $h\in H$.\par
We cite \cite[Theorem 3.6]{NVW07a}:
\begin{theorem}[$L^p$-stochastic integrability]\label{t:stochint}
Let $p\in (1,\infty)$ and $T>0$ be fixed. For an $H$-strongly measurable adapted process $\Phi:(0,T)\times \Omega
\rightarrow \calL(H,X)$ the
following are equivalent:
\begin{enumerate}
\item $\Phi$ is $L^p$-stochastically integrable with respect to $W_H$;\label{l:sintble}
\item we have $\Phi^* x^* \in L^p(\Omega;L^2(0,T;H))$ for all $x^*\in X^*$
and there exists a (necessarily unique) $R_\Phi\in L^{p}(\Omega;
\g(0,T,H;X))$ such that for all  $x^*\in X^*$ we have
$$R_\Phi^* x\s = \Phi^* x\s$$ in $L^p(\O;L^2(0,T;H))$.
\end{enumerate}
In this situation one has
\begin{align}\label{BDG}
\E\, \sup_{0\leq t\leq T} \Big\n \int_{0}^t \Phi \,dW_H \Big\n_Y^p \eqsim_p
\E\, \n R_\Phi \n_{\g(0,T,H;Y)}^{p},
\end{align}
the implied constants being independent of $\Phi$.
\end{theorem}\par
The inequalities in equation \eqref{BDG} are referred to as \emph{Burkholder-Davis-Gundy} inequalities.\par
From now on, if $\Phi$ is stochastically integrable, we shall simply use $\Phi$ to denote both a process and the
(unique) $R_\Phi\in L^{p}(\Omega; \gamma(0,t,H;X))$ that satisfies $R_\Phi^* x\s = \Phi^* x\s$.\par
\subsubsection{The SDE}\label{ss:setting} In Section \ref{sec:app} we will prove a perturbation result for the following
stochastic differential equation:
\begin{align}\label{SDE}\tag{SDE}
\left\{ \begin{array}{rcl} dU(t) &\! =\! & AU(t)\,dt + F(t,U(t))\,dt + G(t,U(t))\,dW_H(t);\quad t\in [0,T],\\
U(0)&\! =\!  &x_0. \end{array}\right.
\end{align}
where $A$, $F$ and $G$ are assumed to satisfy conditions \MA, \MF,
and \MG\ below. The assumptions on $F$ depend on the type of the
Banach space $F$. The type of a Banach space is defined based on
the behavior of moments of randomized sums, and it  always takes
values in the interval $[1,2]$. The greater the type, i.e., the
closer to $2$, the more the space behaves like a Hilbert space:
every Banach space has type at least $1$, and all Hilbert spaces
have type $2$. The $L^p$-spaces have type $\min\{p,2\}$. We refer
to \cite{LeTa:91} for a precise definition of type (and co-type)
and further details. Both \textsc{umd} and type are preserved
under Banach space isomorphisms.\par
\begin{description}
\item[\MA{}] $A$ generates an analytic $C_0$-semigroup on a \textsc{umd} Banach space $X$.
\item[\MF{}] For some $\theta_F>  -1 + (\inv{\tau}-\frac12)$, where $\tau$ is the type of $X$, the function 
$F:[0,T]\times X\rightarrow X_{\theta_F}$ is measurable in the sense that for all $x\in X$ the mapping $F(\cdot,x):[0,T]\rightarrow X_{\theta_F}$ is strongly measurable. Moreover, $F$ is uniformly Lipschitz continuous and uniformly of linear growth on $X$.\par

That is to say, 
there exist constants $C_0$ and
$C_1$ such that for all $t\in [0,T]$ and all $x,y\in X$:
\begin{align*}
\n F(t,x) - F(t,y) \n_{X_{\theta_F}} & \leq C_0 \n x-y\n_{X}, \\ 
\n F(t,x)\n_{X_{\theta_F}} &\leq 
C_1(1+\n x\n_{X}).
\end{align*}
The least constant $C_0$ such that the above holds is denoted by Lip$(F)$, and the least constant $C_1$ such that the
above holds is denoted by $M(F)$.\par
\item[\MG{}] For some $\theta_G>-\inv{2}$, the function $G : [0,T]\times X\rightarrow \calL(H,X_{\theta_G})$ 
is measurable in the sense that for all $h\in H$ and $x\in X$ the mapping $G(\cdot,x)h:[0,T]\rightarrow X_{\theta_G}$ is strongly measurable. Moreover, $G$ is uniformly $L^2_{\gamma}$-Lipschitz continuous and uniformly of linear growth on $X$.\par
That is to say, 
there exist constants $C_0$ and 
$C_{1}$ such that for 
all $\alpha\in [0,\inv{2})$, all $t\in [0,T]$, and all simple functions $\phi_1$, $\phi_2$, $\phi: [0,T]\to X$ one has:
\begin{align*}
&\n s\mapsto (t-s)^{-\alpha}[G(s,\phi_1(s)) -G(s,\phi_2(s))] \n_{\gamma(0,t;H,X_{\theta_G})} \\
&\qquad \qquad \le  C_0 \n s\mapsto (t-s)^{-\alpha}[\phi_1 -\phi_2] 
\n_{L^2(0,t;X)\,\cap \,\gamma(0,t;X)}; \\
&\n s\mapsto (t-s)^{-\alpha} G(s,\phi(s))\n_{\gamma(0,t;H,X_{\theta_G})}\\
&\qquad \qquad  \leq C_1\big(1+ \n s\mapsto
(t-s)^{-\alpha}\phi(s)\n_{L^2(0,t;X)\,\cap\, \gamma(0,t;X)}\big).
\end{align*}% 
The least constant $C_0$ such that the above holds is denoted by $\textrm{Lip}_{\gamma}(G)$, and the least constant $C_1$ such that the
above holds is denoted by $M_{\gamma}(G)$.
\end{description}

If $Y_2$ is a type 2 space and $G:[0,T]\times Y_1\rightarrow \gamma(H,Y_2)$ is Lipschitz-continuous, uniformly in $[0,T]$, then $G$ is $L^2_{\gamma}$-Lipschitz continuous (see \cite[Lemma 5.2]{NVW08}). More examples of
$L^2_{\gamma}$-Lipschitz continuous operators can be found in
\cite{NVW08}.\par

\subsubsection{Existence and uniqueness}

We recall an existence and uniqueness result for the problem \eqref{SDE}. This result is formulated in a space of continuous, `weighted' stochastically integrable processes which is defined as follows:

\begin{definition}\label{d:Vspace}
For $\alpha\geq 0$, $1\leq p<\infty$ and $0\leq a \leq b <
\infty$, we denote by $\Vpc{\alpha}([a,b]\times \Omega;Y)$ the
space of adapted, continuous processes $\Phi:[a,b]\times \Omega\rightarrow
Y$ for which the following norm is finite:
\begin{align*}
\n \Phi\n_{\Vpc{\alpha}([a,b]\times \Omega;Y)} = \n \Phi \n_{L^p(\Omega;C([a,b];Y))} + \sup_{a\leq t\leq b}\n
s\mapsto (t-s)^{-\alpha}\Phi(s)\n_{L^p(\Omega;\gamma(a,t;Y))}.
\end{align*}
\end{definition}

One easily checks that for $a\leq c<d\leq b$,
\begin{equation}\label{Vtransinv}
\n \Phi|_{[c,d]} \n_{ \Vpc{\alpha}([a,b]\times\Omega;X)} = \n \Phi|_{[c,d]} \n_{
\Vpc{\a} ([c,d]\times \Omega;X)}.
\end{equation}\par
Moreover, for $0\le\beta\le\alpha<\frac12$ and $\Phi \in \Vpc{\alpha}([a,b]\times \Omega;X)$ one has:
\begin{equation}\label{Vchange-of-alpha}
\n\Phi\n_{\Vpc{\beta} ([a,b]\times \Omega;X)} \leq (b-a)^{\alpha-\beta}\n
\Phi\n_{\Vpc{\alpha}([a,b]\times \Omega;X)}.
\end{equation}
Note also that we have $\Vpc{\alpha}([0,T]\times \Omega;X)
\subset L^p(\Omega;C([0,T];X))$. On the other hand, we have the following embedding (see \cite[Lemma 3.3]{NVW08}):
\begin{lemma}\label{lem:VEmbedHoelder}
Let $X$ be a Banach space with type $\tau$. Then for all $T>0$, $\eps>0$ and $\alpha
\in [0,\inv{2})$ there exists a constant $C$ such that for all $T_0\in [0,T]$ one has:
\begin{equation}
\Vpc{\alpha}([0,T_0]\times\Omega;X) \leq C L^p(\Omega;C^{\inv{\tau}-\inv{2}+\eps} ([0,T_0];X)).
\end{equation}
\end{lemma}

If $G:[0,T]\times X\rightarrow \calL(H,X_{\theta_G})$ satisfies \MG{} and
$\Phi_1,\Phi_2 \in \Vpc{\alpha}([0,T]\times \Omega;X)$ 
for some $p\geq 2$, then:
\begin{equation}
\begin{aligned}\label{GLipschitzV2}
& \sup_{0\leq t\leq T}\n s\mapsto
(t-s)^{-\alpha}[G(s,\Phi_1(s))-G(s,\Phi_2(s))]\n_{L^p(\Omega;\gamma(0,t;X_{
\theta_G}))}\\
&\qquad \qquad \leq (1+T^{\inv{2}-\alpha})\textrm{Lip}_{\gamma}(G)\n
\Phi_1-\Phi_2\n_{\Vpc{\alpha}([0,T]\times \Omega;X)},
\end{aligned}
\end{equation}
and, for $\Phi \in \Vpc{\alpha}([0,T]\times \Omega;X)$:
\begin{equation}
\begin{aligned}\label{GLipschitzV1}
&\sup_{0\leq t\leq T}\n s\mapsto
(t-s)^{-\alpha}G(s,\Phi(s))\n_{L^p(\Omega;\gamma(0,t;X_{\theta_G}))} \\
&\qquad \qquad  \leq (1+T^{\inv{2}-\alpha})M_\gamma(G)\big( 1+\n \Phi
\n_{\Vpc{\alpha} ([0,T]\times \Omega;X)}\big).
\end{aligned}
\end{equation}
The following existence and uniqueness result for solutions
to \eqref{SDE} is presented in \cite[Theorem 6.2]{NVW08}.

\begin{theorem}[Van Neerven, Veraar and Weis, 2008]\label{thm:NVW08}
Consider \eqref{SDE} under the assumptions \MA{}, \MF{}, and \MG{}. Let $x_0\in L^p(\Omega,\mathcal{F}_0;X_{\eta})$ for $p\in (2,\infty)$ and $\eta>0$ satisfying
$$ 0 \leq \eta < \min\{\tfrac{3}{2}-\tinv{\tau}+\theta_F,\, \tinv{2}-\tinv{p}+\theta_G \}.$$ 
Then for any $T>0$ and any $\alpha\in [0,\inv{2})$ there exists a unique $U \in \Vpc{\alpha}([0,t]\times\Omega;X_{\eta}^{A})$ such that $s\mapsto S(t-s)G(s,U(s))$ is stochastically integrable for all $t\in[0,T]$, and $U$ satisfies:
\begin{align}\label{SDE_sol}
U(t)=S(t)x_0 + \int_{0}^{t} S(t-s)F(s,U(s))\,ds + \int_{0}^{t} S(t-s)G(s,U(s))\,dW_H(s)
\end{align}
almost surely for all $t\in [0,T]$. Moreover:
\begin{align}\label{UinVestimate}
\n U \n_{ \Vpc{\alpha}([0,T]\times \Omega;X_{\eta}^{A})}  & \lesssim 1+\n x_0\n_{L^p(\Omega;X_{\eta}^{A})}.
\end{align}
\end{theorem}
\begin{remark}
In \cite{NVW08} the authors assume $\theta_F\leq 0$ and $\theta_G\leq 0$ (and promptly refer to them as $-\theta_F$ and
$-\theta_B$). However, one may check that the theorem remains valid for $\theta_F,\theta_G\geq 0$, which leads to extra
space regularity of the solution (i.e., greater values for $\eta$ in $\eqref{UinVestimate}$).\par
Moreover, in \cite{NVW08} the authors assume $\alpha>\inv{p}-\theta_G$, this assumption can be dropped: existence of a
solution in $\Vpc{\alpha}([0,t]\times\Omega;X)$ for any $\alpha\in [0,\inv{2})$ follows by equation \eqref{Vchange-of-alpha}. Uniqueness of a
solution for any $\alpha\in [0,\inv{2})$ follows by observing that if $\Phi \in \Vpc{\alpha}([0,t]\times\Omega;X_{\eta}^{A})$ for
some $\alpha\in [0,\inv{2})$, and $\Phi$ satisfies \eqref{SDE_sol}, then $\Phi \in
\Vpc{\beta}([0,t]\times\Omega;X_{\eta}^{A})$ for $0\leq \beta<\alpha+\inv{2}+\theta_G-\eta$ (this follows by the proof of Proposition \ref{prop:stochConvV}).
\end{remark}\par
\subsection{\texorpdfstring{$\gamma$-Boundedness}{Randomized boundedness}}
For vector-valued stochastic integrals, the concept of $\gamma$-boundedness plays the same role as uniform boundedness does for ordinary integrals: the Kalton-Weis multiplier theorem (Proposition \ref{prop:KW} below) allows one to estimate terms out of a stochastic integral, provided they are $\gamma$-bounded. Note that any $\gamma$-bounded set of operators is automatically uniformly bounded, and the reverse holds if $X$ is a Hilbert space.\par
A family $\mathscr{B}\subset \calL(X,Y)$ is called {\em
$\g$-bounded} if there exists a constant $C$ such that for all
$N\ge 1$, all $x_1,\dots, x_N\in X$, and all $B_1,\dots, B_N\in \mathscr{B}$ we have:
$$\E \Big\n \sum_{n=1}^N \g_n B_n x_n\Big\n_Y^2 \le C^2 \E \Big\n \sum_{n=1}^N \g_n x_n\Big\n_X^2.
$$
The least admissible constant $C$ is called the $\g$-bound of $\mathscr{B}$,
notation: $\g_{[X,Y]}(\mathscr{B})$.\par
The following lemma is a direct consequence of the Kahane contraction principle:
\begin{lemma}\label{lem:cont}
If $\mathscr{B}\subset \calL(X,Y)$ is $\g$-bounded and $M>0$ then $M\mathscr{B}:= \{ aB: a\in [-M,M], B\in
\mathscr{B}\}$ is $\g$-bounded with $\g_{[X,Y]}(M\mathscr{B}) \leq M \g_{[X,Y]}(\mathscr{B})$.
\end{lemma}\par
The following proposition, which is a variation of a result of Weis \cite[Proposition 2.5]{Weis01}, gives a sufficient condition for $\g$-boundedness.
\begin{proposition}\label{prop:integr-der}
Let $f:[0,T]\rightarrow \calL(X,Y)$
be a function such that for all $x\in X$ the function $t\mapsto f(t)x$ is
continuously differentiable. Suppose $g\in L^1(0,T)$ is such that for all $t\in (0,T)$:
\begin{align*}
\n \tfrac{d}{dt}f(t)x \n_Y & \leq g(t)\n x\n_X,\quad \textrm{for all } x\in X.
\end{align*}
Then the set $\mathscr{R}: = \{f(t): \ t\in (0,T)\}$ is
$\g$-bounded in $\mathscr{L}(X,Y)$ and
$$ \g_{[X,Y]}(\mathscr{R}) \le \n f(0+)\n + \n g\n_{L^1(0,T)}.$$
\end{proposition}\par
The following $\g$-multiplier result, due to Kalton
and Weis \cite{KalWei07} (see also \cite{Nee-survey}),
establishes a relation between stochastic integrability
and $\g$-bounded\-ness.
\begin{proposition}[$\gamma$-Multiplier theorem]\label{prop:KW}
Suppose $X$ does not contain a closed subspace
isomorphic to $c_0$.
Suppose $M:(0,T)\to \calL(X,Y)$ is a
strongly measurable function
with $\gamma$-bounded range $\mathcal{M}=\{M(t): \ t\in (0,T)\}$.
If $\Phi\in  \g(0,T,H;X)$ then $M\Phi\in \g(0,T,H;Y)$ and:
$$ \n M\Phi\n_{\g(0,T,H;Y)} \leq \g_{[X,Y]}(\mathcal{M})\,
\n \Phi\n_{\g(0,T,H;X)}.$$
\end{proposition}
\begin{remark}
The assumption that $X$ should not contain a copy of $c_0$ can be avoided, provided one replaces
$\g(0,T,H;X)$ by $\g_\infty(L^2(0,T;H),X)$, the space of all $\g$-summing operators
from $L^2(0,T;H)$ to $X$. We refer to \cite{Nee-survey} for more details.
In all applications in this paper, $X$ is a \textsc{umd} space and therefore
does not contain a copy of $c_0$.
\end{remark}\par
Finally, we recall the following $\gamma$-boundedness estimate for analytic semigroups (see e.g.\ \cite[Lemma
4.1]{NVW08}).
\begin{lemma}\label{lem:analyticRbound} Let $X$ be a Banach space and let $A$ be the generator of an analytic
$C_0$-semigroup $S$ of type $(\omega,\theta,K)$ on $X$.
Then for all $0\leq \delta<\alpha$ and $T>0$ there exists a constant $C$ depending on $S$ only in terms of $\omega$, $\theta$, and $K$, such that for all $t\in (0,T]$ the set
$\mathscr{S}_{\alpha,t} = \{s^\alpha S(s): \ s\in [0,t]\}$ is $\g$-bounded
in $\calL(X,X_\delta^A)$ and we have $$\gamma_{[X,X_\delta^A]}(\mathscr{S}_{\alpha,t})\leq C t^{\alpha-\delta}, \quad t\in (0,T].$$
\end{lemma}
Note that the constant $C$ in the lemma above \emph{may} depend on $T$.

\subsection{Estimates for (stochastic) convolutions}
In this section we provide the estimates for (stochastic) convolutions necessary to derive the perturbation result given in Theorem \ref{t:app}. In order to avoid confusion further on, we shall use $Y_1$ and $Y_2$ to denote \textsc{umd} Banach spaces in this section. \par
The following lemma is proven in \cite{coxNeer:11}. It is an adaptation of \cite[Proposition 4.5]{NVW08}.
\begin{lemma}\label{lem:h1}
Let $(R,\mathcal{R},\mu)$ be a finite measure space and $(S,\mathcal{S},\nu)$ a
$\sigma$-finite measure space.   
Let $\Phi_1:[0,T]\times \Omega \rightarrow \calL(H,Y_1)$, let $\Phi_2\in
L^1(R;\calL(Y_1,Y_2))$, 
and let  $f\in L^{\infty}(R\times [0,T];L^2(S))$. If $\Phi_1$ is
$L^p$-stochastically integrable for some $p\in (1,\infty)$, then
\begin{align*}
& \Big\n  s\mapsto \int_{0}^{T} \int_R f(r,u)(s)\Phi_2(r)\Phi_1(u)\, d\mu(r) 
\,dW_H(u)\Big\n_{L^p(\Omega;\gamma(S;Y_2))}\\
& \qquad \qquad\lesssim  \esssup_{(r,u)\in R\times
[0,T]} \n f(r,u)\n_{L^2(S)} \n \Phi_2 \n_{L^1(R,\calL(Y_1,Y_2))}  \n \Phi_1
\n_{L^p(\Omega;\gamma(0,T;H,Y_1))},
\end{align*}
with implied depending only on  $p$, $Y_1$, $Y_2$, provided the right-hand side
is finite.
\end{lemma}\par
To our knowledge,
most regularity results for stochastic convolutions are based on the factorization method introduced in
\cite{DaKwaZab87}. The result below is merely based on the regularity of the convolving functions.
\begin{lemma}\label{lem:stochConv}
Let $T>0$, $p\in [1,\infty)$ and $\eta> 0$. Suppose the process $\Phi\in L^p(\Omega;\gamma(0,T,H;Y_1))$ is
adapted to $(\calF_t)_{t\geq 0}$ and satisfies:
\begin{align*}
\sup_{0\leq t\leq T}\n s\mapsto (t-s)^{-\eta}\Phi(s)\n_{L^p(\Omega;\gamma(0,t,H;Y_1))}<\infty.
\end{align*}\par
Let $\Psi:[0,T]\rightarrow \calL(Y_1,Y_2)$ be such that $\Psi x$ is continuously differentiable on $(0,T)$ for all $x\in
Y_1$. Suppose moreover there exists a $g\in L^{1}(0,T)$ and $0 \leq  \theta < \eta$ such that
\begin{align*}
  v^\theta \n \tfrac{d}{dv}\Psi(v) x\n_{Y_2} + \theta v^{\theta-1} \n  \Psi(v) x\n_{Y_2}  & \leq g(v)\n x\n_{Y_1}, \quad \textrm{for all } x\in Y_1.
\end{align*}
Then the stochastic convolution process $$t\mapsto \int_{0}^{t}\Psi(t-s)\Phi(s)\,dW_H(s)$$ is well-defined and
\begin{align*}
& \Big\n t\mapsto \int_{0}^{t} \Psi(t-s)\Phi(s)\,dW_H(s) \Big\n_{C^{\eta-\theta}([0,T];L^p(\Omega;Y_2))} \\
& \qquad \qquad  \leq  2\bar{C}_p \n g \n_{L^{1}(0,T)}\sup_{0\leq t\leq T}\n s\mapsto
(t-s)^{-\eta}\Phi(s)\n_{L^p(\Omega;\gamma(0,t,H;Y_1))},
\end{align*}
where $\bar{C}_p$ is the constant in the Burkholder-Davis-Gundy inequality for the $p^{\textrm{\scriptsize{th}}}$ moment, for the space $Y_1$, see equation \eqref{BDG}.
\end{lemma}\par
Before proving the Lemma, observe that the corollary below follows directly from Kolmogorov's continuity criterion
(see Theorem I.2.1 in Revuz and Yor).
\begin{corollary}\label{cor:stochConv}
Let the setting be as in Lemma \ref{lem:stochConv} and assume in addition that $\inv{p}<\eta-\theta$. Let $0<\beta
< \eta-\theta-\inv{p}$. There exists a modification of the
stochastic convolution process $t\mapsto \int_{0}^{t}\Psi(t-s)\Phi(s)\,dW_H(s)$, which we shall denote by
$\Psi\diamond \Phi$, such that:
\begin{align*}
& \n \Psi\diamond \Phi \n_{L^p(\Omega;C^{\beta}([0,T];Y_2))} \\
& \qquad \qquad\leq \tilde{C} \n g \n_{L^{1}(0,T)}\sup_{0\leq t\leq T}\n
s\mapsto
(t-s)^{-\eta}\Phi(s)\n_{L^p(\Omega;\gamma(0,t,H;Y_1))},
\end{align*}
where $\tilde{C}$ depends only on $\eta$, $\beta$ and $p$ and $\bar{C}_p$.\par
\end{corollary}
\begin{proof}[Proof of Lemma \ref{lem:stochConv}]
By Proposition \ref{prop:integr-der} and assumption it follows that $\{s^{\theta}\Psi(s):s\in [0,T] \}$ is
$\gamma$-bounded. Thus by the Kalton-Weis multiplier Theorem (see Proposition \ref{prop:KW}), and the fact that
$$\sup_{0\leq t\leq T}\n s\mapsto (t-s)^{-\eta}\Phi(s)\n_{L^p(\Omega;\gamma(0,t,H;Y_1))}<\infty,$$ it follows that
$s\mapsto \Psi(t-s)\Phi(s)1_{s\in[0,t]}\in L^p(\Omega;\gamma(0,t,H;Y_2))$ for all $t\in [0,T]$. By Theorem \ref{t:stochint} this process
is stochastically integrable.\par
In what follows we let $\frac{d}{dv}$ denote the derivative with respect to the strong operator topology. By the triangle inequality we have:
\begin{equation}
\begin{aligned}\label{hoelderLp}
&\Big\n \int_{0}^{t} \Psi(t-u)\Phi(u)\,dW_H(u)  -  \int_{0}^{s} \Psi(s-u)\Phi(u)\,dW_H(u) \Big\n_{L^p(\Omega;Y_2)}\\
& \qquad \qquad \leq \Big\n \int_{0}^{s} [\Psi(t-u)-\Psi(s-u)]\Phi(u)\,dW_H(u) \Big\n_{L^p(\Omega;Y_2)} \\
& \qquad \qquad \quad + \Big\n \int_{s}^{t} \Psi(t-u)\Phi(u)\,dW_H(u) \Big\n_{L^p(\Omega;Y_2)}\\
& \qquad\qquad  = \Big\n \int_{0}^{s} \int_{s-u}^{t-u}\tfrac{d}{dv}\Psi(v) \,dv \Phi(u)\,dW_H(u) \Big\n_{L^p(\Omega;Y_2)}\\
& \qquad \qquad \quad + \Big\n \int_{s}^{t} (t-u)^{-\theta}\int_{0}^{t-u}\tfrac{d}{dv}[v^{\theta}\Psi(v) ]\,dv \Phi(u)\,dW_H(u)
\Big\n_{L^p(\Omega;Y_2)}.
\end{aligned}
\end{equation}
We now wish to apply the stochastic Fubini theorem (see \cite[Lemma 2.7]{CoxGor:11}, \cite{NV:06}). Consider $\Upsilon:[0,s]\times [0,t]\rightarrow \calL(H,Y)$ defined by $\Upsilon(u,v)= 1_{\{s-u\leq v\leq t-u\}}\tfrac{d}{dv}\Psi(v)\Phi(u)$. As $\tfrac{d}{dv}\Psi$ is strongly continuous and $\Phi$ is $H$-strongly measurable, we have that $\Upsilon$ is $H$-strongly measurable. Moreover, as $\Phi$ is adapted it follows that $\Upsilon_{v}:=\Upsilon(\cdot,v)$ is adapted for almost all $v\in [0,t]$. Finally, we have that $\Upsilon\in L^1(0,t;\gamma(0,s,H;Y_2))$ by assumption:
\begin{align*}
\n \Upsilon(\cdot,v)\n_{\gamma(0,s,H;Y_2)} \leq v^{\eta-\theta}g(v)\n u\mapsto (s-u)^{-\eta}\Phi(u)\n_{\gamma(0,s,H;Y_1)},
\end{align*}
where we use that $v\geq s-u$ on $\textrm{supp}(\Upsilon)$.\par
Note that stochastic Fubini theorem in \cite[Lemma 2.7]{CoxGor:11}, \cite{NV:06} requires $\Upsilon_v$ to be progressive. However, it suffices to assume that $\Upsilon_v$ is adapted, see \cite{Ver:11}.
Thus the conditions necessary to apply the stochastic Fubini theorem are satisfied, and we have:
\begin{equation}\label{mainReg1}
\begin{aligned}
&\Big\n \int_{0}^{s} \int_{s-u}^{t-u} \tfrac{d}{dv}[\Psi(v)\Phi(u)]\,dv\,dW_H(u) \Big\n_{L^p(\Omega;Y_2)}\\
& \quad = \Big\n \int_{0}^{t} \int_{(s-v)\maxsym 0}^{(t-v)\minsym s}\tfrac{d}{dv}[\Psi(v)\Phi(u)]\,dW_H(u)
\,dv\Big\n_{L^p(\Omega;Y_2)}\\
& \quad \leq \int_0^{t} v^{-\theta} g(v) \Big\n \int_{(s-v)\maxsym 0}^{(t-v)\minsym s} \Phi(u)\,dW_H(u)
\Big\n_{L^p(\Omega;Y_1)}\,dv\\
& \quad \leq \bar{C}_p \int_0^{t} v^{-\theta} g(v) \big\n 1_{[(s-v)\maxsym 0,(t-v)\minsym s]}\Phi
\big\n_{L^p(\Omega;\gamma(0,t; Y_1))}\,dv\\
& \quad \leq \bar{C}_p \int_{0}^{t} v^{-\theta} g(v) [(t-s)\minsym v]^{\eta}\big\n u\mapsto ([((t-v)\minsym
s)-u]^{-\eta}\Phi(u))\big\n_{L^p(\Omega;\gamma(0,t; Y_1))}\,dv\\
& \quad \leq \bar{C}_p (t-s)^{\eta-\theta} \int_{0}^{t} g(v)\,dv \sup_{t\in [0,T]}\big\n u\mapsto
(t-u)^{-\eta}\Phi(u)\big\n_{L^p(\Omega;\gamma(0,t; Y_1))}.
\end{aligned}
\end{equation}
For the final term in \eqref{hoelderLp} one may also check that the conditions of the stochastic Fubini hold and thus:
\begin{align*}
&  \Big\n \int_{s}^{t} (t-u)^{-\theta}\int_{0}^{t-u}\tfrac{d}{dv}[v^{\theta}\Psi(v)\Phi(u)]\,dv\,dW_H(u)
\Big\n_{L^p(\Omega;Y_2)}\\
& \qquad \qquad \leq \int_{0}^{t-s} g(v) \Big\n \int_{s}^{t-v} (t-u)^{-\theta}\Phi(u)\,dW_H(u) \Big\n_{L^p(\Omega;Y_1)}\,dv\\
& \qquad \qquad \leq \bar{C}_p \int_{0}^{t-s}  g(v) \big\n u\mapsto 1_{[s,t-v]}(u) (t-u)^{-\theta}\Phi(u)
\big\n_{L^p(\Omega;\gamma(0,t,H;F_1))}\,dv\\
& \qquad \qquad \leq \bar{C}_p (t-s)^{\eta-\theta} \n g\n_{L^1(0,T)} \sup_{t\in [0,T]}\big\n u\mapsto
(t-u)^{-\eta}\Phi(u)\big\n_{L^p(\Omega;\gamma(0,t; Y_1))}.
\end{align*}
By inserting the two estimates above in \eqref{hoelderLp} we obtain that
\begin{align*}
&\Big\n \int_{0}^{t} \Psi(t-u)\Phi(u)\,dW_H(u) - \int_{0}^{s} \Psi(s-u)\Phi(u)\,dW_H(u) \Big\n_{L^p(\Omega;Y_2))}\\
& \qquad\qquad  \leq 2\bar{C}_p(t-s)^{\eta-\theta}\n g\n_{L^1(0,T)} \sup_{t\in [0,T]}\big\n u\mapsto
(t-u)^{-\eta}\Phi(u)\big\n_{L^p(\Omega;\gamma(0,t; Y_1))},
\end{align*}
which completes the proof as $0\leq s<t\leq T$ where chosen arbitrarily.
\end{proof}
\begin{remark}
In the setting of the lemma above one may also take $g\in L^{q'}(0,T)$ and $\Phi$ such that $$\int_{0}^{T}\big\n
s \mapsto (t-s)^{-\eta}\Phi(s)\big\n_{L^p(\Omega;\gamma(0,t; Y_1))}^q dt<\infty,$$ where $q\in [1,\infty]$;
$\inv{q}+\inv{q'}=1$. In that case one obtains:
\begin{align*}
& \Big\n t\mapsto \int_{0}^{t} \Psi(t-s)\Phi(s)\,dW_H(s) \Big\n_{C^{\eta-\theta}([0,T];L^p(\Omega;Y_2))} \\
& \qquad \qquad \leq (3+2^{\eta})\bar{C}_p \n g \n_{L^{q'}(0,T)} \Big(\int_{0}^{T}\big\n
s\mapsto (t-s)^{-\eta}\Phi(s)\big\n_{L^p(\Omega;\gamma(0,t,H;Y_1))}^{q}\,dt\Big)^{\inv{q}}.
\end{align*}
We omit the proof because it requires significantly more space, and we do not need this result in what follows (for the extended proof, see \cite[Lemma A.9]{cox:thesis}).
\end{remark}\par
Based on the above two lemmas, we obtain the following result for stochastic convolutions in the $\Vpc{\alpha}$-norm:
\begin{proposition}\label{prop:stochConvV}
Let the setting be as in Lemma \ref{lem:stochConv} and assume in addition that $\inv{p}< \eta-\theta$. Let $\alpha\in [0,\inv{2})$.
Then $\Psi\diamond \Phi\in \Vpc{\alpha}([0,T]\times\Omega;Y_2)$. Moreover, there exists a constant $C$ such that for all $T_0\in [0,T]$ we have:
\begin{align*}
\n \Psi\diamond \Phi \n_{\Vpc{\alpha}([0,T_0]\times\Omega;Y_2)} & \leq C  \n g \n_{L^{1}(0,T)} \sup_{0\leq t\leq
T_0}\n s\mapsto (t-s)^{-\eta}\Phi(s)\n_{L^p(\Omega;\gamma(0,t;Y_1))}.
\end{align*}
\end{proposition}
\begin{proof}
For the norm estimate in $L^p(\Omega;C([0,T_0];Y_2))$ we apply Corollary \ref{cor:stochConv} with $\beta>0$ such that
$0\leq \beta+\theta < \eta - \inv{p}$. For the estimate in the weighted $\gamma$-norm fix $t\in [0,T_0]$. We
apply Lemma \ref{lem:h1} with $\Phi_1(u)=(t-u)^{-\eta}\Phi(u)1_{0\leq u < t}$, $\Phi_2(r) = \frac{d}{dr}[r^{\theta}\Psi(r)]$, $R=[0,t]$
and $f(r,u)(s)= (t-s)^{-\alpha}(s-u)^{-\theta}(t-u)^{\eta}1_{0\leq r< s-u}1_{0\leq u < t}$. From Lemma \ref{lem:h1} it follows that:
\begin{align*}
& \Big\n s\mapsto (t-s)^{-\alpha}\int_{0}^{s}\Psi(s-u)\Phi(u)\,dW_H(u)\Big\n_{L^p(\Omega;\gamma(0,t;Y_2))} \\
&\qquad\qquad  \lesssim t^{\inv{2}+\eta-\alpha-\theta}\n g\n_{L^1(0,T)}\n s\mapsto
(t-s)^{-\eta}\Phi(s)\n_{L^p(\Omega;\gamma(0,t;Y_1))}.
\end{align*}
Taking the supremum over $t\in[0,T_0]$ and using that $\eta-\theta>0$ (whence $T_0^{\inv{2}+\eta-\alpha-\theta}\leq T^{\inv{2}+\eta-\alpha-\theta}$) we arrive at the desired result.
\end{proof}\par
For deterministic convolutions we have the following:
\begin{proposition}\label{prop:detHolder}
Suppose $\Phi\in L^p(\Omega;L^{\infty}(0,T;Y_1))$ for some $p\in [1,\infty)$. Let $\Psi:[0,T]\rightarrow \calL(Y_1,Y_2)$ be such that $\Psi x$ is
continuously differentiable on $(0,T)$ for all $x\in Y_1$ Suppose moreover there exists a $g\in L^{1}(0,T)$ and a $\theta\in [0,1]$ such that
for all $v\in (0,T)$ we have:
\begin{align*}
v^\theta \n \tfrac{d}{dv}\Psi(v) x\n_{Y_2} + \theta v^{\theta-1} \n \Psi(v) x\n_{Y_2}  & \leq g(v)\n x\n_{Y_1}, \quad \textrm{for all } x\in Y_1.
\end{align*}
Then there exists a constant $C$ such that for all $T_0\in [0,T]$ we have, almost surely:
\begin{align*}
\n \Psi*\Phi\n_{C^{1-\theta} ([0,T_0];Y_2)} & \leq C \n g \n_{L^1(0,T_0)} \n
\Phi\n_{L^{\infty}(0,T_0;Y_1)}.
\end{align*}
\end{proposition}\par
By Lemma \ref{lem:VEmbedHoelder}, with $\eps=\frac{3}{2}-\inv{\tau}T-\theta$, we obtain the following corollary:
\begin{corollary}\label{cor:detConvV}
Let the setting be as in Proposition \ref{prop:detHolder}. Assume in addition that $Y_2$ has type $\tau$, and let  $0\leq \theta < \frac{3}{2}-\inv{\tau}$. Then for $\alpha\in [0,\inv{2})$ and $p\in [1,\infty)$ there exists a constant $C$ such that for $T_0 \in [0,T]$ one has:
\begin{align*}
\n \Psi * \Phi \n_{\Vpc{\alpha}([0,T_0]\times\Omega;Y_2)} \leq C \n g \n_{L^{1}(0,T_0)} \n \Phi\n_{L^p(\Omega;L^{\infty}(0,T_0;Y_1))}.
\end{align*}
\end{corollary}
\begin{proof}[Proof of Proposition \ref{prop:detHolder}]
Observe that we have, for $0\leq s< t\leq T_0$:
\begin{equation}
\begin{aligned}\label{hoelderC}
&\Big\n \int_{0}^{t} \Psi(t-u)\Phi(u,\omega)\,du  -  \int_{0}^{s} \Psi(s-u)\Phi(u,\omega)\,du \Big\n_{Y_2} \\
& \qquad  \qquad \leq \Big\n \int_{0}^{s} \int_{s-u}^{t-u} \tfrac{d}{dv}[\Psi(v)\Phi(u,\omega)]\,dv\,du \Big\n_{Y_2}\\
& \qquad \qquad\quad + \Big\n \int_{s}^{t} (t-u)^{-\theta}\int_{0}^{t-u}\tfrac{d}{dv}[v^{\theta}\Psi(v) \Phi(u,\omega)]\,dv
\,du
\Big\n_{Y_2}.
\end{aligned}
\end{equation}
Now
\begin{equation*}
\begin{aligned}
& \Big\n \int_{0}^{s} \int_{s-u}^{t-u}  \tfrac{d}{dv}[\Psi(v)\Phi(u,\omega)]\,dv\,du \Big\n_{Y_2}\\
& \qquad \qquad \leq \int_{0}^{s} \int_{(s-v)\maxsym s}^{(t-v)\minsym s} \,du (s-v)^{-\theta}g(v) \,dv \n \Phi(\omega)\n_{L^{\infty}(0,t;Y_1)}\\
& \qquad \qquad \leq (t-s)^{1-\theta}\int_{0}^{t} g(v) \,dv \n \Phi(\omega)\n_{L^{\infty}(0,t;Y_1)},
\end{aligned}
\end{equation*}
where we used that $\int_{(s-v)\maxsym s}^{(t-v)\minsym s} \,du \leq (t-s)\minsym v$.
Furthermore, we have
\begin{equation*}
\begin{aligned}
&\Big\n \int_{s}^{t} (t-u)^{-\theta}\int_{0}^{t-u}\tfrac{d}{dv}[v^{\eps}\Psi(v)] \Phi(u,\omega)\,dv \,du \Big\n_{Y_2} \\
&\qquad \qquad \leq
(1-\theta)^{-1}(t-s)^{1-\theta}\Big\n \int_{0}^{t} g(v) \,dv \n \Phi(\omega)\n_{L^{\infty}(0,T;Y_1)}.
\end{aligned}
\end{equation*}
Inserting these two estimates in \eqref{hoelderC} completes the proof.
\end{proof}
\section{A perturbation result}\label{sec:app}
In this section we shall prove the perturbation theorem announced in the introduction. Consider \eqref{SDE} with $A$, $F$ and $G$ satisfying \MA, \MF, and \MG. Keeping in mind possible applications in approximations of
solutions to stochastic partial differential equations, we let $X_0$ be a (possibly finite-dimensional) closed subspace
of $X$. We assume there exists a bounded projection $P_0:X\rightarrow X_0$ such that $P_0(X)=X_0$. Let $i_{X_0}$
represent the canonical embedding of $X_0$ into $X$ (note however that we shall omit $i_{X_0}$ when it is clear from
the context).\par
Let $A_0$ be the generator of an
analytic $C_0$-semigroup $S_0$ on $X_0$. For $t\geq 0$ define $\tilde{S}_0\in \calL(X)$ by $
\tilde{S}_0(t):=i_{X_0}S_0(t) P_0,$ this defines a \emph{degenerate} $C_0$-semigroup, i.e., $\tilde{S}_0$
satisfies the semigroup property but $\tilde{S}_0(0)=i_{X_0}P_0$ (which is clearly not the identity unless
$X_0=X$).\par
Let $x_0\in L^p(\Omega,\calF_0;X)$ (where
$p>2$ satisfies $\inv{p}\leq \inv{2}+\theta_G$) and let $U$ be the solution to \eqref{SDE} as provided by Theorem \ref{thm:NVW08}.\par
\begin{theorem}\label{t:app}
Let $\omega\geq 0$, $\theta\in(0,\frac{\pi}{2})$ and $K>0$ be such that $A$ and $A_0$ are both of type
$(\omega,\theta,K)$. Suppose there exist $\delta\in [0,1]$ and $p\in (2,\infty)$ satisfying
$$0\leq \delta < \min\{\tfrac{3}{2}-\tinv{\tau}+\theta_F,\, \tinv{2}-\tinv{p}+\theta_G\}$$
such that for some  $\lambda_0\in \rho(A)$ we have:
\begin{align}\label{app:ass}
D_{\delta}(A,A_0):=  \n R(\lambda_0:A)-i_{X_0} R(\lambda_0:A_0)P_0 \n_{\calL(X_{\delta-1}^{A},X)} < \infty.
\end{align}
Suppose $x_0\in L^p(\Omega,\calF_0;X_{\delta}^{A})$ and $y_0 \in L^p(\Omega,\calF_0;X)$.\par
For any $\alpha\in [0,\inv{2})$ there exists a unique process $U^{(0)}\in \Vpc{\alpha}([0,T_0]\times\Omega;X_0)$
such that $s\mapsto 1_{[0,t]}S_0(t-s)P_0G(s,U^{(0)}(s))$ is stochastically integrable for all $t\in [0,T]$ and for
all $t\in [0,T]$ we have:
\begin{equation}
\begin{aligned}\label{pert.sol}
U^{(0)}(t) = S_0(t-s)P_0 y_0 & + \int_{0}^{T} S_0(t-s)P_0F(s,U^{(0)}(s))\,ds\\
& + \int_{0}^{t}
S_0(t-s)P_0 G(s,U^{(0)}(s))\,dW_H(s), \quad \textrm{a.s.}
\end{aligned}
\end{equation}
Moreover:
\begin{equation}
\begin{aligned}\label{pert_est}
& \n U - i_{X_0}U^{(0)} \n_{\Vpc{\alpha}([0,T]\times \Omega;X)} \\
& \qquad\qquad \lesssim  \n x_0 - y_0 \n_{L^p(\Omega;X)} + D_{\delta}(A,A_0)(1+\n x_0 \n_{L^p(\Omega;X_{\delta}^{A})}).
\end{aligned}
\end{equation}
The implied constant depends on $X_0$ only in terms of $\n P_0\n_{\calL(X,X_0)}$, on $A$ and $A_0$ only in terms of $1+D_{\delta}(A,A_0)$, $\omega$, $\theta$ and $K$, and on $F$ and
$G$ only in terms of their Lipschitz and linear growth constants Lip$(F)$, $\textrm{Lip}_{\gamma}(G)$, $M(F)$ and $M_{\gamma}(G)$.
\end{theorem}\par
To prove Theorem \ref{t:app} we need a proposition concerning the $\gamma$-boundedness of $S-\tilde{S}_0$. The proof of
this proposition is postponed to the end of this section.
\begin{proposition} \label{p:gbddSG}
Let $A$, $A_0$ be as introduced above, i.e., $A$ generates an analytic semigroup on $X$ and $A_0$ generates an analytic
semigroup on $X_0$. Let $\omega\geq 0, \theta\in (0,\frac{\pi}{2})$ and $K>0$ be such that $A$ and $A_0$ are of type
$(\omega,\theta,K)$. Suppose there exists a $\lambda_0\in \C$, $\mathscr{R}e(\lambda_0) > \omega$, and $\delta\in \R$
such that $D_{\delta}(A,A_0)<\infty$, where $D_{\delta}(A,A_0)$ is as defined in \eqref{app:ass}. Set $$\omega'=\omega+|\lambda_0-\omega|(\cos\theta)^{-1}.$$ Then for all
$\beta \in \R$ such that $\beta\in [\delta-1,\delta]$ one has:
\begin{align}\label{p:gbddsg:unif1}
\sup_{t\in [0,\infty)} t^{\delta-\beta} e^{-\omega' t} \n S(t) - \tilde{S}_0(t)\n_{\calL(X_{\beta}^{A},X)} & \lesssim D_{\delta}(A,A_0),
\end{align}
and
\begin{align}\label{p:gbddsg:unif2}
\sup_{t\in [0,\infty)} t^{\delta-\beta+1}e^{-\omega' t} \n \tfrac{d}{dt}S(t)-\tfrac{d}{dt}\tilde{S}_0(t) \n_{\calL (X_{\beta}^{A},X)}
\lesssim D_{\delta}(A,A_0),
\end{align}
with implied constants depending only on $\n P_0\n_{\calL(X,X_0)}$, $\omega$, $\theta$, $K$, $\delta-\beta$.\par
Moreover, for all $\alpha > \delta-\beta$ we have, for $t\in [0,T]$:
\begin{align*}
\gamma_{[X_{\beta}^{A},X]}\left(\left\{  s^{\alpha} [ S(s) - \tilde{S}_0(s)]; \, 0\leq s\leq t\right\}\right) &
\lesssim t^{\alpha+\beta-\delta} D_{\delta}(A,A_0),
\end{align*}
with implied constant depending only on $\n P_0\n_{\calL(X,X_0)}$, $\omega$, $\theta$, $K$, $\delta-\beta$, and $T$.
\end{proposition}
\begin{proof}[Proof of Theorem \ref{t:app}.]
We split the proof into several parts.
\subsection*{Part 1.}
In order to prove existence and uniqueness of $U^{(0)}\in
\Vpc{\alpha}([0,T_0]\times\Omega;X_{0})$ satisfying \eqref{pert.sol} it suffices, by Theorem
\ref{thm:NVW08}, to prove that there exist $\eta_{F}>-\frac{3}{2}+\inv{\tau}$ and $\eta_{G}>-\inv{2}+\inv{p}$ such that
$P_0 F:[0,T]\times X\rightarrow X_{0,\eta_{F}}^{A_0}$ is Lipschitz continuous and of linear growth and $P_0
G:[0,T]\times X\rightarrow
\gamma(H,X_{0,\eta_{G}}^{A_0})$ is $L^2_{\gamma}$-Lipschitz continuous and of linear growth. If $\theta_F\geq 0$ then clearly we may take
$\eta_{F}=0$, and we have Lip$(P_0F)\leq \n P_0\n_{\calL(X,X_0)}\textrm{Lip}(F)$, $M(P_0F)\leq \n P_0\n_{\calL(X,X_0)}M(F)$. The same goes for $\theta_G\geq 0$.\par
Now suppose $\theta_F<0$. Recall the following representation of negative fractional powers of an operator $A$ generating an analytic semigroup $S$ of type $(\omega,\theta,K)$ (see \cite[Chapter
2.6]{Pazy:83}):
$$ (\lambda I-A)^{\eta}=\inv{\Gamma(-\eta)}\int_0^{\infty}t^{-\eta-1}e^{-\lambda t}S(t)dt, \quad  \eta<0,\
\mathscr{R}e(\lambda)>\omega.$$
Let $\bar{\omega}>\omega'$, where $\omega'$ is as in Proposition \ref{p:gbddSG}. From the representation above and Proposition \ref{p:gbddSG} it follows that for $\beta\in [\delta-1,\delta]$,
$\eta<\beta-\delta$ and $x\in X$ we have:
\begin{align*}
& \n P_0 x \n_{X_{0,\eta}^{A_0}} \eqsim \n ((\bar{\omega} I - A_0)^{\eta}P_0 x \n_{X} = \Big\n
\inv{\Gamma(-\eta)}\int_0^{\infty}t^{-\eta-1}e^{-\bar{\omega} t}\tilde{S}_0(t) x dt\Big\n_{X}\\
&\qquad   \leq \inv{\Gamma(-\eta)}\int_0^{\infty}t^{-\eta-1}e^{-\bar{\omega} t}\n (S(t)-\tilde{S}_0(t))x\n_{X}dt  \\
&\qquad  \quad + \inv{\Gamma(-\eta)}\Big\n \int_0^{\infty}t^{-\eta-1}e^{-\bar{\omega} t}S(t) x dt \Big\n_{X}\\
&\qquad \lesssim D_{\delta}(A,A_0)\int_0^{\infty}t^{-\eta-1+\beta-\delta}e^{-(\bar{\omega}-\omega') t} dt \n x \n_{X_{\beta}^{A}} + \n (\bar{\omega} I -
A)^{\eta }x\n_X,
\end{align*}
with implied constants depending on $X_0$ only in terms of $\n P_0\n_{\calL(X,X_0)}$ and on $A$ and $A_0$ only in terms of $\omega,\theta,$ and $K$. Thus for $\beta\in [\delta-1,\delta]$, $\eta<\beta-\delta$ we have:
\begin{equation}
\begin{aligned}\label{fracspaceEmbed}
\n P_0 x \n_{X_{0,\eta}^{A_0}} & \eqsim \n (\bar{\omega} I - A_0)^{\eta}P_0 x \n_{X} \\
& \lesssim (1+D_{\delta}(A,A_0))\n (\bar{\omega} I - A)^{\beta} x \n_{X} \eqsim (1+D_{\delta}(A,A_0))\n x \n_{X_{\beta}^{A}},
\end{aligned}
\end{equation}
with implied constants depending on $X_0$ only in terms of $\n P_0\n_{\calL(X,X_0)}$ and on $A$ and $A_0$ only in terms of $\omega,\theta,$ and $K$. \par
Note that by assumption we have $\theta_F> -\frac{3}{2}+\inv{\tau} + \delta \geq \delta -1$. Hence one can pick
$\eta_{F}$ such that
$-\frac{3}{2}+\inv{\tau}<\eta_{F}<\theta_F-\delta$. By \eqref{fracspaceEmbed} it follows that $P_0 F:[0,T]\times
X\rightarrow
X_{0,\eta_{F}}^{A_0}$ is Lipschitz continuous and 
\begin{align}\label{eq:LipP0F}
\textrm{Lip}(P_0F)& \lesssim (1+D(A,A_0))\textrm{Lip}(F); & M(P_0F)\lesssim (1+D(A,A_0))M(F),
\end{align}
with implied constant depending on $X_0$ only in terms of $\n P_0\n_{\calL(X,X_0)}$ and on $A$ and $A_0$ only in terms of $\omega,\theta,$ and $K$.\par

Similarly, if
$\theta_G<0$ there exists a
$\eta_{G}$ such that $-\inv{2}+\inv{p}<\eta_{G}<\theta_G-\delta$ such that $P_0 G:[0,T]\times X\rightarrow
\gamma(H,X_{0,\eta_{G}}^{A_0})$ is $L^2_{\gamma}$-Lipschitz continuous and 
\begin{align}\label{eq:LipP0G}
\textrm{Lip}_{\gamma}(P_0G)& \lesssim (1+D(A,A_0))\textrm{Lip}_{\gamma}(G); & M_{\gamma}(P_0G)\lesssim (1+D(A,A_0))M_{\gamma}(G),
\end{align}
with implied constant depending on $X_0$ only in terms of $\n P_0\n_{\calL(X,X_0)}$ and on $A$ and $A_0$ only in terms of $\omega,\theta,$ and $K$.\par

\subsection*{Part 2.}
Define $\tilde{U}^{(0)} = i_{X_0}U^{(0)}$ and observe that if $U^{(0)}$ satisfies \eqref{pert.sol},
then $\tilde{U}^{(0)}$ satisfies:
\begin{align*}
\tilde{U}^{(0)}(t) = \tilde{S}_0(t-s) y_0 &+ \int_{0}^{T} \tilde{S}_0(t-s)F(s,\tilde{U}^{(0)}(s))\,ds\\
& + \int_{0}^{t} \tilde{S}_0(t-s)G(s,\tilde{U}^{(0)}(s))\,dW_H(s), \quad \textrm{a.s.}
\end{align*}
Let $T_0\in [0,T]$ be fixed. By the above we have:
\begin{equation}\label{pert.split}
\begin{aligned}
& \n U - \tilde{U}^{(0)}\n_{\Vpc{\alpha}([0,T_0]\times\Omega;X)} \\
& \qquad \leq \n (S-\tilde{S}_0)x_0
\n_{\Vpc{\alpha}([0,T_0]\times\Omega;X)} + \n \tilde{S}_0(x_0-y_0) \n_{\Vpc{\alpha}([0,T_0]\times\Omega;X)}\\
& \qquad\quad + \Big\n t\mapsto \int_0^{t}\tilde{S}_0(t-s)[F(s,U(s))-F(s,\tilde{U}^{(0)}(s))]\,ds
\Big\n_{\Vpc{\alpha}([0,T_0]\times\Omega;X)}\\
& \qquad\quad + \Big\n t\mapsto \int_0^{t}[S(t-s)-\tilde{S}_0(t-s)]F(s,U(s))\,ds
\Big\n_{\Vpc{\alpha}([0,T_0]\times\Omega;X)}\\
& \qquad\quad + \Big\n t\mapsto \int_0^{t}\tilde{S}_0(t-s)[G(s,U(s))-G(s,\tilde{U}^{(0)}(s))]\,dW_H(s)
\Big\n_{\Vpc{\alpha}([0,T_0]\times\Omega;X)}\\
& \qquad\quad + \Big\n t\mapsto \int_0^{t}[S(t-s)-\tilde{S}_0(t-s)]G(s,U(s))\,dW_H(s)
\Big\n_{\Vpc{\alpha}([0,T_0]\times\Omega;X)}.
\end{aligned}
\end{equation}\par
Let $\eta_{F}$ and $\eta_{G}$ be as defined in part 1. Let $\eps>0$ be such that
\begin{align*}
\eps&\leq 1-2\alpha;\\
\eps&<\min\{\tfrac{3}{2}-\tinv{\tau}+\eta_{F},\tinv{2}-\tinv{p}+\eta_{G}\}.
\end{align*}
It follows that $\eps + \delta < \min\{\tfrac{3}{2}-\tinv{\tau}+\theta_{F},\tinv{2}-\tinv{p}+\theta_{G}\}$.
By equation \eqref{Vchange-of-alpha} we may
assume, without loss of generality, that $\alpha=\inv{2}-\eps/2$.\par
We will estimate each of the six terms on the right-hand side of \eqref{pert.split} in parts 2a-2f below. In part 2c and
2e we keep track of the dependence on $T_0$, for the other parts this is not necessary.
\subsubsection*{Part 2a.}
By Proposition \ref{p:gbddSG} with $\beta=\delta$ there exists an $\calM>0$ depending on $X_0$ only in terms of $\n P_0\n_{\calL(X,X_0)}$, and on $A$ and $A_0$ only in terms
om $\omega$, $\theta$ and $K$, such that
\begin{align*}
\sup_{t\in [0,T_0]}\n S(t) - \tilde{S}_0(t)\n_{\calL(X_{\delta}^{A},X)}\leq \calM D_{\delta}(A,A_0);\\
\gamma_{[X_{\delta}^{A},X]}\{t^{\eps/2}(S(t)-\tilde{S}_0(t))\,:\,t\in[0,T_0]\}\leq \calM D_{\delta}(A,A_0).
\end{align*}
Thus by Proposition \ref{prop:KW} we have
\begin{align*}
& \n (S-\tilde{S}_0)x_0\n_{\Vpc{\alpha}([0,T_0]\times\Omega;X)} \\
& \quad \leq \calM  D_{\delta}(A,A_0)\big[\sup_{t\in[0,T_0]}\n s\mapsto
(t-s)^{-\alpha}s^{-\eps/2}x_0\n_{L^p(\Omega;\gamma(0,t;X_{\delta}^{A}))}+ \n x_0\n_{L^p(\Omega;X_{\delta}^{A})}\big].
\end{align*}
For $f\in L^2(0,t)$ and $x\in L^p(\Omega;X_{\delta}^{A})$ we have $$\n f\otimes
x\n_{L^p(\Omega;\gamma(0,t;X_{\delta}^{A}))}=\n f\n_{L^2(0,t)}\n x\n_{L^p(\Omega;X_{\delta}^{A})}.$$ Thus, recalling
that $\alpha=\inv{2}-\eps/2$, we have:
\begin{align*}
&\sup_{t\in[0,T_0]}\n s\mapsto
(t-s)^{-\alpha}s^{-\eps/2}x_0\n_{L^p(\Omega;\gamma(0,t;X_{\delta}^{A}))} \\
&\qquad \qquad  \leq \n s\mapsto
(1-s)^{-\alpha}s^{-\eps/2} \n_{L^2(0,1)} \n x_0\n_{L^p(\Omega;X_{\delta}^{A})} \leq C_{\eps} \n x_0\n_{L^p(\Omega;X^A_{\delta})},
\end{align*}
where $C_{\eps}$ is a constant depending only on $\eps$, and we used that $\alpha=\inv{2}-\eps/2$.
Hence
\begin{equation}\label{pert.h1}
\begin{aligned}
\n (S-\tilde{S}_0)x_0\n_{\Vpc{\alpha}([0,T_0]\times\Omega;X)} & \leq \calM D_{\delta}(A,A_0)(1+C_{\eps})\n
x_0\n_{L^p(\Omega;X_{\delta}^{A})}.
\end{aligned}
\end{equation}
\par
\subsubsection*{Part 2b.}
By assumption (see Remark \ref{r:analcont}) there exists an $\calM$ depending only on $\n P_0\n_{\calL(X,X_0)}$, $\omega,\theta,$ and $K$ and $T$ such that
we have that $\sup_{t\in [0,T]}\n \tilde{S}_0(t)\n_{\calL(X,X_0)} \leq \calM$. Moreover, by Lemma \ref{lem:analyticRbound}
we may pick $\calM$ such that in addition we have that $\gamma_{[X,X]}\{t^{\eps/2}\tilde{S}_0(t)\,:\,t\in[0,T_0]\}\leq
\calM.$ Thus by the same argument as in part 2a we have:
\begin{equation}\label{pert.h2}
\begin{aligned}
\n \tilde{S}_0(x_0-y_0) \n_{\Vpc{\alpha}([0,T_0]\times\Omega;X)} & \leq \calM (1+C_{\eps}) \n x_0 -y_0 \n_{L^p(\Omega;X)}.
\end{aligned}
\end{equation}\par
\subsubsection*{Part 2c.}
Recall that $\eta_{F}\leq 0$. By equation \eqref{analyticSG-Est} there exists an $\calM$ depending
only on $\omega$, $\theta$, $K$ and $T$ such that for all $t\in [0,T]$ we have:
\begin{align*}
& t^{-\eta_{F}+\eps} \n \tfrac{d}{dt}S_0(t) x\n_{\calL(X_{0,\eta_{F}}^{A_0},X)} + (\eps-\eta_{F}) t^{-\eta_{F}+\eps-1} \n  S_{0}(t) x\n_{\calL(X_{0,\eta_{F}}^{A_0},X)} \\
& \qquad \qquad= t^{-\eta_{F}+\eps} \n \tfrac{d}{dt}S_0(t) x\n_{\calL(X_{0,\eta_{F}}^{A_0},X_0)} + (\eps-\eta_{F}) t^{-\eta_{F}+\eps-1} \n  S_{0}(t) x\n_{\calL(X_{0,\eta_{F}}^{A_0},X_0)} \\
& \qquad \qquad\leq \calM t^{-1+\eps}.
\end{align*}
By Corollary \ref{cor:detConvV} with $Y_1=X_{0,\eta_{F}}^{A_0}$, $Y_2=X$, $$\Phi(s) =
P_0[F(s,U(s))-F(s,\tilde{U}^{(0)}(s))],$$ $\Psi(s)=S_0(s)$, $\theta = -\eta_{F}+\eps$ and $g(v)=\calM
v^{-1+\eps}$,  it follows that:
\begin{equation}\label{pert.h3}
\begin{aligned}
&\Big\n t\mapsto \int_0^{t}\tilde{S}_0(t-s)[F(s,U(s))-F(s,\tilde{U}^{(0)}(s))]\,ds
\Big\n_{\Vpc{\alpha}([0,T_0]\times\Omega;X)}\\
& \qquad \qquad\lesssim T_0^{\eps}\n
P_0[F(\cdot,U)-F(\cdot,\tilde{U}^{(0)})]\n_{L^p(\Omega;L^{\infty}(0,T_0;X_{0,\eta_{F}}^{A_0}))}\\
& \qquad \qquad\leq T_0^{\eps}\textrm{Lip}(P_0F)\n
U-\tilde{U}^{(0)}\n_{L^p(\Omega;L^{\infty}(0,T_0;X))}\\
& \qquad \qquad\lesssim T_0^{\eps}(1+D_{\delta}(A,A_0))\textrm{Lip}(F)\n
U-\tilde{U}^{(0)}\n_{L^p(\Omega;L^{\infty}(0,T_0;X))},
\end{aligned}
\end{equation}
where the second-last estimate follows by Lipschitz-continuity of $P_0F$ and the final estimate follows by \eqref{eq:LipP0F}. Note that the implied constants are independent of $T_0$, and depend on $X_0$ only in terms of $\n P_0\n_{\calL(X,X_0)}$ and on $A_0$ only
in terms of $\omega$, $\theta$ and $K$.\par

\subsubsection*{Part 2d.}
By Proposition \ref{p:gbddSG} with $\beta= \theta_F\minsym \delta\in[\delta-1,\delta]$ we have that there exists a
constant $\calM$ depending only on $\n P_0\n_{\calL(X,X_0)}$, $\omega,\theta,K$, $(\delta-\theta_F)\maxsym 0$ and $T$ such that for all $t\in[0,T]$ we
have: 
\begin{align*}
& t^{(\delta-\theta_F)^+ +\eps}\n \tfrac{d}{dt}[S(t)-\tilde{S}_0(t)]\n_{\calL(X_{\theta_F\minsym
\delta}^{A},X)}\\
& \qquad +((\delta-\theta_F)^+ +\eps)t^{(\delta-\theta_F)^+ +\eps-1}\n S(t)-\tilde{S}_0(t)\n_{\calL(X_{\theta_F\minsym
\delta}^{A},X)}  \leq \calM D_{\delta}(A,A_0) t^{-1+\eps}.
\end{align*}
Thus by Corollary \ref{cor:detConvV} with $Y_1=X_{\theta_F\minsym
\delta}^{A}$,
$Y_2=X$, $\Phi(s)= F(s,U(s))$, $\Psi(s)= S(s)-\tilde{S}_0(s)$, $\theta= (\delta-\theta_F)^{+}+\eps$, $g(v)=\calM
D_{\delta}(A,A_0) v^{-1+\eps}$,  we obtain:
\begin{equation}\label{pert.h4}
\begin{aligned}
&\Big\n t\mapsto \int_0^{t}[S(t-s)-\tilde{S}_0(t-s)]F(s,U(s))\,ds \Big\n_{\Vpc{\alpha}([0,T_0]\times\Omega;X)}\\
&\qquad\qquad \lesssim D_{\delta}(A,A_0)\n F(\cdot,U)\n_{L^p(\Omega;L^{\infty}(0,T_0;X_{\theta_F\minsym \delta}^{A}))}\\
&\qquad \qquad\leq D_{\delta}(A,A_0)M(F)\n
U\n_{L^p(\Omega;L^{\infty}(0,T_0;X))} \\
&\qquad \qquad\lesssim D_{\delta}(A,A_0)M(F)(1+\n x_0\n_{L^p(\Omega;X)}),
\end{aligned}
\end{equation}
where the penultimate estimate follows by the linear growth condition on $F$ and the final estimate by
\eqref{UinVestimate}.
Note that the implied constants are independent of $T_0$, and depend on $X_0$ only in terms of $\n P_0\n_{\calL(X,X_0)}$, and on $A$
and $A_0$ only in terms of $\omega$, $\theta$ and $K$.\par
\subsubsection*{Part 2e.}
Recall that $\eta_{G}\leq 0$. By equation \eqref{analyticSG-Est} there exists an $\calM$ depending only on $\omega$, $\theta$, $K$ and $T$ such that for all $t\in [0,T]$ we have:
\begin{align*}
& t^{-\eta_{G}+\eps/2} \n \tfrac{d}{dt} S_0(t)\n_{\calL(X_{0,-\eta_{G}}^{A},X_0)}+ (\eps/2-\eta_{G})t^{-\eta_{G}+\eps/2-1} \n S_0(t)\n_{\calL(X_{0,-\eta_{G}}^{A},X_0)} \\
& \qquad \qquad \leq \calM t^{-1+\eps/2}.
\end{align*}
By applying Proposition \ref{prop:stochConvV} with $Y_1=X_{0,\eta_{G}}^{A_0}$, $Y_2=X$, $\Psi(s)= S_0(s)$, $\eta=\alpha$, $\alpha=\alpha$,
$\theta=-\eta_{G}+\eps/2$ and $g(v)= \calM v^{-1+\eps/2}$ and $$\Phi(s)=
P_0[G(s,U(s))-G(s,\tilde{U}^{(0)}(s))]$$ we obtain:
\begin{equation}\label{pert.h5}
\begin{aligned}
&\Big\n t\mapsto \int_0^{t}\tilde{S}_0(t-s)[G(s,U(s))-G(s,\tilde{U}^{(0)}(s))]\,dW_H(s)
\Big\n_{\Vpc{\alpha}([0,T_0]\times\Omega;X)}\\
&\qquad \lesssim T_0^{\inv{2}\eps}\sup_{0\leq t\leq T_0}\n s\mapsto
(t-s)^{-\alpha}P_0[G(s,U(s))-G(s,\tilde{U}^{(0)}(s))]\n_{L^p(\Omega;\gamma(0,t;X_{0,\eta_{G}}^{A_0}))}\\
&\qquad \lesssim T_0^{\inv{2}\eps}(1+D_{\delta}(A,A_0))\textrm{Lip}_{\gamma}(G)\n
U-\tilde{U}^{(0)}\n_{\Vpc{\alpha}([0,T_0]\times\Omega;X)},
\end{aligned}
\end{equation}
where the final estimate follows from  estimates \eqref{GLipschitzV2} and \eqref{eq:LipP0G}. Note that the implied constants are independent of $T_0$, and depend on $X_0$ only in terms of $\n P_0\n_{\calL(X,X_0)}$, and on $A$ and $A_0$ only in terms of $\omega$, $\theta$ and $K$.\par

\subsubsection*{Part 2f.}
By Proposition \ref{p:gbddSG} with $\beta= \theta_G\minsym \delta\in [\delta-1,\delta]$ we have that there exists a
constant $\calM $ depending only on $\n P_0\n_{\calL(X,X_0)}$, $\omega,\theta,K$, $(\delta-\theta_G)\maxsym 0$ and $T$ such that for all $t\in[0,T]$ we
have: 
\begin{align*}
& t^{(\delta-\theta_G)^{+}+\eps/2}\n \tfrac{d}{dt} [S(t)-\tilde{S}_0(t)]\n_{\calL(X_{\theta_G\minsym
\delta}^{A},X)}\\
& \qquad \qquad+ ((\delta-\theta_G)^{+}+\eps/2)t^{(\delta-\theta_G)^{+}+\eps/2-1}\n S(t)-\tilde{S}_0(t)\n_{\calL(X_{\theta_G\minsym
\delta}^{A},X)} \\
& \qquad \leq \calM D_{\delta}(A,A_0)t^{-1+\eps/2}.
\end{align*}
Thus by Proposition \ref{prop:stochConvV} with
$Y_1=X_{\theta_G\minsym \delta}^{A}$, $Y_2=X$, $\Phi(s)= G(s,U(s))$, $\Psi(s)= S(s)-\tilde{S}_0(s)$, $\eta=\alpha$,
$\alpha=\alpha$, $\theta=(\delta-\theta_G)^+ +\eps/2$ and $g(v)= \calM D_{\delta}(A,A_0)v^{-1+\eps/2}$ we obtain:
\begin{equation}\label{pert.h6}
\begin{aligned}
&\Big\n t\mapsto \int_0^{t}[S(t-s)-\tilde{S}_0(t-s)]G(s,U(s))\,dW_H(s) \Big\n_{\Vpc{\alpha}([0,T_0]\times\Omega;X)}\\
&\qquad\qquad \lesssim D_{\delta}(A,A_0)\sup_{0\leq t\leq T_0}\n s\mapsto
(t-s)^{-\alpha}G(s,U(s))\n_{L^p(\Omega;\gamma(0,t;X_{\theta_G\minsym \delta}^{A}))}\\
&\qquad\qquad \leq D_{\delta}(A,A_0)M_{\gamma}(G)\n
U\n_{\Vpc{\alpha}([0,T_0]\times\Omega;X)}\\
&\qquad\qquad \lesssim D_{\delta}(A,A_0)M_{\gamma}(G)\big(1+\n x_0
\n_{L^p(\Omega;X)}\big),
\end{aligned}
\end{equation}
where the penultimate line follows by estimate \eqref{GLipschitzV1}. Note that the implied constants
are independent of $T_0$, and depend on $X_0$ only in terms of $\n P_0\n_{\calL(X,X_0)}$ and on $A$ and $A_0$ only in terms of $\omega$, $\theta$ and $K$.
\subsubsection*{Part 2g.}
Inserting \eqref{pert.h1}-\eqref{pert.h6} in \eqref{pert.split} we obtain that there exists a constant
$C>0$ independent of $x_0$ and $y_0$, depending on $X_0$ only in terms of $\n P_0\n_{\calL(X,X_0)}$, on $A$ and $A_0$ only in terms of $1+D_{\delta}(A,A_0)$, $\omega$, $\theta$ and $K$, and on $F$ and
$G$ only in terms of their Lipschitz and linear growth constants Lip$(F)$, $\textrm{Lip}_{\gamma}(G)$, $M(F)$ and $M_{\gamma}(G)$, such that for all $T_0\in [0,T]$ one has:
\begin{align*}
& \n U - \tilde{U}^{(0)}\n_{\Vpc{\alpha}([0,T_0]\times\Omega;X)} \\
& \qquad \qquad \leq CT_{0}^{\inv{2}\eps}\n U - U^{(0)}\n_{\Vpc{\alpha}([0,T_0]\times\Omega;X)} \\
& \qquad \qquad \quad  + C\Big(\n x_0-y_0 \n_{L^p(\Omega;X)}+D_{\delta}(A,A_0)\big(1+\n x_0\n_{L^p(\Omega;X_{\delta}^{A})}\big)\Big).
\end{align*}
Setting $T_0= [2C]^{-2/\eps}$ we obtain:
\begin{equation}\label{T0est}
\begin{aligned}
& \n U - \tilde{U}^{(0)}\n_{\Vpc{\alpha}([0,T_0]\times\Omega;X)} \\
& \qquad \qquad \leq 2C\Big(\n x_0-y_0
\n_{L^p(\Omega;X)}+D_{\delta}(A,A_0)\big(1+\n x_0\n_{L^p(\Omega;X_{\delta}^{A})}\big)\Big).
\end{aligned}
\end{equation}

\subsection*{Part 3.}
Let $t_0\geq 0$, $z\in L^p(\Omega,\calF_{t_0};X)$, $T>0$ and $\alpha\in [0,\inv{2})$. By $U(z,t_0,\cdot)$, we denote
the (unique) process in $\Vpc{\alpha}([t_0,t_0+T]\times\Omega;X)$ satisfying, for $s\in [t_0,t_0+T]$:
\begin{align*}
U(z,t_0,s)=S(t-t_0)z& +\int_{t_0}^{t}S(t-t_0-s)F\big(U(z,t_0,s)\big)\,ds \\
&+ \int_{t_0}^{t}S(t-t_0-s)G\big(U(z,t_0,s)\big)\,dW_H(s) \quad \textrm{a.s.}
\end{align*}
The process $U^{(0)}(z,t_0,\cdot)$ is defined analogously.\par
From the proof of \eqref{T0est} it follows that for any $x\in L^p(\Omega,\calF_{t_0};X_{\delta}^{A})$ and $y\in
L^p(\Omega,\calF_{t_0};X)$ we have:
\begin{equation}\label{T0est.shift}
\begin{aligned}
&\n U(x,t_0,\cdot) - U^{(0)}(y,t_0,\cdot)\n_{\Vpc{\alpha}([t_0,t_0+T_0]\times\Omega;X)} \\
&\qquad\qquad  \leq 2C\Big(\n x-y \n_{L^p(\Omega;X)}+D_{\delta}(A,A_0)[1+\n x\n_{L^p(\Omega;X_{\delta}^{A})}]\Big),
\end{aligned}
\end{equation}
with $C$ as in \eqref{T0est}.

\subsection*{Part 4.}
Throughout this section one may check that the implied constants always depend $X_0$ only in terms of $\n P_0\n_{\calL(X,X_0)}$, on $A$ and $A_0$ only in terms of $1+D_{\delta}(A,A_0)$, $\omega$, $\theta$ and $K$, and on $F$ and $G$ only in terms of their Lipschitz and linear growth constants Lip$(F)$, $\textrm{Lip}_{\gamma}(G)$, $M(F)$ and $M_{\gamma}(G)$, even when this is not mentioned explicitly.\par
By uniqueness of the solution to \eqref{SDE} it follows that for any $0\leq s_0\leq t_0\leq t$ and any $x,y\in
L^p(\Omega,\calF_{s_0};X)$ one has:
$$ U(x,s_0,t)=U\big(U(x,s_0,t_0),t_0,t\big) \ \ \textrm{and}\ \
U^{(0)}(y,s_0,t)=U^{(0)}\big(U^{(0)}(y,s_0,t_0),t_0,t\big).$$\par
Let $j\in \N$. By the embedding $\Vpc{\alpha}([0,T_0]\times\Omega;X) \hookrightarrow L^{\infty}(0,T_0;L^p(\Omega;X))$ and
estimate
\eqref{T0est.shift} with $x=U(x_0,0,(j-1)T_0)$ and $y=U^{(0)}(y_0,0,(j-1)T_0)$, we obtain:
\begin{equation}\label{recur2}\begin{aligned}
& \n U(x_0,0,jT_0)-U^{(0)}(y_0,0,jT_0)\n_{L^p(\Omega;X)} \\
& \qquad \qquad= \Big\n U\Big(U\big(x_0,0,(j-1)T_0\big),(j-1)T_0,T_0\Big)\\
& \qquad \qquad\qquad\quad -U^{(0)}\Big(U^{(0)}\big(y_0,0,(j-1)T_0\big),(j-1)T_0,T_0\Big) \Big\n_{L^p(\Omega;X)}\\
& \qquad \qquad\lesssim \n U\big(x_0,0,(j-1)T_0\big)-U^{(0)}\big(y_0,0,(j-1)T_0\big)\n_{L^p(\Omega;X)}\\
& \qquad \qquad\quad + D_{\delta}(A,A_0)\big[1+\n U\big(x_0,0,(j-1)T_0\big)\n_{L^p(\Omega;X^{A}_{\eta})}\big],\end{aligned}
\end{equation}
with implied constant independent of $j$ and $n$ and the `initial values' $U(x_0,0,(j-1)T_0)$ and
$U^{(0)}(y_0,0,(j-1)nT_0)$.\par
By equation \eqref{UinVestimate} it follows that:
\begin{equation}
\begin{aligned}\label{supLpest}
\sup_{1\leq j\leq \lfloor T/T_0 \rfloor} \n U(x_0,0,jT_0) \n_{L^p(\Omega;X^{A}_{\eta})} & \leq \sup_{s\in [0,T]}\n
U(x_0,0,s) \n_{L^p(\Omega;X^{A}_{\eta})}\\
 &  \lesssim 1+\n x_0 \n_{L^p(\Omega;X^{A}_{\eta})}.
\end{aligned}
\end{equation}
\par
Thus for $j=1,\ldots,\lceil T/T_0 \rceil$ we obtain the following relation from \eqref{recur2}:
\begin{align*}
& \n U(x_0,0,jT_0)-U^{(0)}(y_0,0,jT_0)\n_{L^p(\Omega;X)} \\
& \qquad \qquad  \lesssim \n U(x_0,0,(j-1)T_0)-U^{(0)}(y_0,0,(j-1)T_0)\n_{L^p(\Omega;X)}\\
& \qquad \qquad \quad + D_{\delta}(A,A_0)(1+\n x_0 \n_{L^p(\Omega;X^{A}_{\eta})}).
\end{align*}
Note that $T_0$ depends on $X_0$ only in terms of $\n P_0\n_{\calL(X,X_0)}$, on $A$ and $A_0$ only in terms of $1+D_{\delta}(A,A_0)$, $\omega$, $\theta$ and $K$, and on $F$ and $G$ only in terms
of Lip$(F)$, $\textrm{Lip}_{\gamma}(G)$, $M(F)$, and
$M_{\gamma}(G)$. By induction we obtain, for $j=1,\ldots,\lceil T/T_0 \rceil$:
\begin{equation}\label{recur3}\begin{aligned}
&\n U(x_0,jT_0)-U^{(0)}(y_0,jT_0)\n_{L^p(\Omega;X)} \\
& \qquad\qquad  \lesssim \n x_0 - y_0 \n_{L^p(\Omega;X)} + D_{\delta}(A,A_0)(1+\n x_0
\n_{L^p(\Omega;X^{A}_{\eta})}).\end{aligned}
\end{equation}\par
Fix $j\in \N,$ $j< \lceil T/T_0\rceil$. Set $$x=U(x_0,0,(j-1)T_0)\quad \textrm{and} \quad
y=U^{(0)}(y_0,0,(j-1)T_0)$$ in \eqref{T0est.shift} to obtain, using \eqref{supLpest} and \eqref{recur3}:
\begin{align*}
&\big\n U\big(U(x_0,0,(j-1)T_0),(j-1)T_0,\cdot\big)\\
& \qquad \qquad \quad -U^{(0)}\big(U^{(0)}(y_0,(j-1)T_0),(j-1)T_0,\cdot \big)\big\n_{\Vpc{\alpha}([(j-1)T_0,jT_0]\times
\Omega;X)} \\
& \qquad \qquad\lesssim   \n U(x_0,0,(j-1)T_0)-U^{(0)}(y_0,0,(j-1)T_0)\n_{L^p(\Omega;X)} \\
&\qquad\qquad \quad + D_{\delta}(A,A_0) \big(1+\n U(x_0,0,(j-1)T_0) \n_{L^p(\Omega;X^{A}_{\eta})}\big)\\
& \qquad \qquad\lesssim \n x_0 - y_0 \n_{L^p(\Omega;X)} + D_{\delta}(A,A_0)(1+\n x_0 \n_{L^p(\Omega;X^{A}_{\eta})}),
\end{align*}
with implied constants independent of $j$.\par
Due to inequality \eqref{Vtransinv} we thus obtain:
\begin{align*}
&\n U(x_0,0,\cdot) - U^{(0)}(y_0,0,\cdot) \n_{\Vpc{\alpha}([0,T]\times \Omega;X)} \\
& \qquad \leq \Big\n \sum_{j=1}^{\lceil T/T_0 \rceil} \Big( U\big[U\big(x_0,0,(j-1)T_0\big),(j-1)T_0,\cdot\big]\\
& \qquad \qquad \quad -U^{(0)}\big[U^{(0)}\big(y_0,0,(j-1)T_0\big),(j-1)T_0,\cdot \big] \Big)1_{[(j-1)T_0,jT_0)}
\Big\n_{\Vpc{\alpha}([0,T]\times \Omega;X)}\\
& \qquad  \leq   \sum_{j=1}^{\lceil T/T_0 \rceil} \Big\n U\big[U\big(x_0,0,(j-1)T_0\big),(j-1)T_0,\cdot\big]\\
& \qquad \qquad \quad -U^{(0)}\big[U^{(0)}\big(y_0,0,(j-1)T_0\big),(j-1)T_0,\cdot \big]
\Big\n_{\Vpc{\alpha}([(j-1)T_0,jT_0]\times \Omega;X)}\\
&\qquad \lesssim \sum_{j=1}^{\lceil T/T_0 \rceil} \n x_0 - y_0 \n_{L^p(\Omega;X)} +D_{\delta}(A,A_0)(1+\n x_0
\n_{L^p(\Omega;X^{A}_{\eta})})\\
&\qquad \lesssim \n x_0 - y_0 \n_{L^p(\Omega;X)} + D_{\delta}(A,A_0)(1+\n x_0 \n_{L^p(\Omega;X^{A}_{\eta})}).
\end{align*}
\end{proof}\par
It remains to provide a proof for Proposition \ref{p:gbddSG}. For that purpose, we first prove the following lemma.
Given the lemma, the proof of Proposition \ref{p:gbddSG} basically follows the lines of known proofs concerning
comparison of semigroups, see \cite[Chapter III.3.b]{EngNa:00}. For
notational simplicity we define the pseudo-resolvent
\begin{align}\label{pseudoR2}
R(\lambda:\tilde{A}_0)&:=i_{X_0}R(\l:A_0)P_0,\quad \l\in \omega+\Sigma_{\frac{\pi}{2}+\theta}
\end{align}
(we leave it to the reader to verify the resolvent identity).
\begin{lemma} \label{l:extend}
Let the setting be as in Proposition \ref{p:gbddSG}. Then
for all $\l\in \omega'+\Sigma_{\frac{\pi}{2}+\theta}$ we have:
\begin{equation*}
\begin{aligned}
&\n R(\l:A)  - R(\l:\tilde{A}_0)  \n_{\calL(X^{A}_{\beta},X)} \\
& \qquad \qquad \leq C_{\omega,\theta,K,P_0} |\lambda-\omega|^{\delta-\beta-1} \n R(\lambda_0:A)-R(\l_0:\tilde{A}_0)
\n_{\calL(X_{\delta-1}^{A},X)},
\end{aligned}
\end{equation*}
where $C_{\omega,\theta,K,P_0}$ is a constant depending only on $\omega,\theta,K$ and $\n P_0\n_{\calL(X,X_0)}$.
\end{lemma}
\begin{proof}
Using only the resolvent identity and the definition of $R(\lambda:\tilde{A}_0)$ (see \eqref{pseudoR2}) one may verify that the
following identity holds:
\begin{equation}\label{resid}
\begin{aligned}
& R(\l:A) - R(\l:\tilde{A}_0) \\
&\qquad\qquad
=[I+(\lambda_0-\lambda)R(\lambda:\tilde{A}_0)][R(\lambda_0:A)-R(\lambda_0:\tilde{A}_0)](\lambda_0-A)R(\lambda:A).
\end{aligned}
\end{equation}
Moreover, one may check that:
\begin{align*}
\omega'+\Sigma_{\frac{\pi}{2}+\theta}\subset \big(\omega+\Sigma_{\frac{\pi}{2}+\theta}\big)\bigcap \big\{\lambda\in \C:
|\lambda-\omega|\geq |\lambda_0-\omega|\big\}.
\end{align*}
Therefore one has, for $\lambda \in \omega'+\Sigma_{\frac{\pi}{2}+\theta}$, $$\n
I+ (\lambda_0-\lambda) R(\lambda:\tilde{A}_0)\n_{\calL(X)} \leq 1 +
\tfrac{|\lambda_0-\lambda|}{|\lambda-\omega|}K\n P_0\n_{\calL(X,X_0)} \leq 1+2K\n P_0\n_{\calL(X,X_0)}.$$
From \eqref{resid} one obtains:
\begin{equation}\label{Ressplit}
\begin{aligned}
& \n R(\l:A)  - R(\l:\tilde{A}_0)  \n_{\calL(X^{A}_{\beta},X)} \leq (1+2K\n P_0\n_{\calL(X,X_0)})\\
& \qquad \qquad \times \n R(\lambda_0:A)-R(\lambda_0:\tilde{A}_0) \n_{\calL(X_{\delta-1}^{A},X)} \n (\lambda_0-A)R(\lambda:A) \n_{\calL(X^{A}_{\beta},X^{A}_{\delta-1})}.
\end{aligned}
\end{equation}
Let $\bar{\lambda}\in \C$, $\Re e(\bar{\lambda})>\omega$ be such that $|\bar{\lambda}-\lambda_0 | \leq 2| \bar{\lambda} - \omega | $ (if $\Re e(\lambda_0)>\omega$ one may simply pick $\bar{\lambda}=\lambda_0$). For $\eta\in\R$ and $x\in X_{\eta}^{A}$ set $\n x_0\n_{X_{\eta}^A}:= \n (\bar{\lambda}-A)^{\eta}x\n_{X}.$ Then:
\begin{align*}
\n (\lambda_0-A)R(\lambda:A) \n_{\calL(X^{A}_{\beta},X^{A}_{\delta-1})} & = \n (\bar{\lambda}-A)^{\delta-\beta-1}(\lambda_0-A)R(\lambda:A) \n_{\calL(X)}\\
& \leq \big(1+\tfrac{|\bar{\lambda}-\lambda_0|}{|\bar{\lambda}-\omega|}K\big)\n (\bar{\lambda}-A)^{\delta-\beta}R(\lambda:A) \n_{\calL(X)}\\
& \leq (1+2K)\n (\bar{\lambda}-A)^{\delta-\beta}R(\lambda:A) \n_{\calL(X)}.
\end{align*}
\par
If $\delta-\beta =1$ then:
\begin{align*}
\n (\bar{\lambda}-A)^{\delta-\beta}R(\lambda:A) \n_{\calL(X)} & = \n (\bar{\lambda}-A)R(\lambda:A) \n_{\calL(X)}
\leq 1+2K.
\end{align*}
If $\delta-\beta=0$ then:
\begin{align*}
\n (\bar{\lambda}-A)^{\delta-\beta}R(\lambda:A) \n_{\calL(X)} = \n R(\lambda:A) \n_{\calL(X)} \leq
K|\lambda-\omega|^{-1}.
\end{align*}
For $\delta-\beta\in (0,1)$ we have, by Theorem \ref{t:fracPowInt},
\begin{align*}
\n (\bar{\lambda}-A)^{\delta-\beta}R(\lambda:A) \n_{\calL(X)}
& \leq 2(1+K)\n R(\lambda:A) \n_{\calL(X)}^{1+\beta-\delta} \n (\bar{\lambda}-A) R(\lambda:A)
\n_{\calL(X)}^{\delta-\beta}\\
& \leq 2(1+K)(1+2K)^{\delta-\beta} K^{1+\beta-\delta}|\lambda - \omega |^{\delta-\beta-1}\\
&  \leq 2(1+2K)^2|\lambda - \omega |^{\delta-\beta-1}.
\end{align*}
Substituting this into \eqref{Ressplit} one obtains:
\begin{align*}
& \n R(\l:A)  - R(\l:\tilde{A}_0)  \n_{\calL(X^{A}_{\beta},X)} \\
& \qquad \leq  2(1+2K)^4 \n P_0\n_{\calL(X,X_0)}|\lambda - \omega |^{\delta-\beta-1}\n R(\lambda_0:A)-R(\lambda_0:\tilde{A}_0)
\n_{\calL(X_{\delta -1}^{A},X)}.
\end{align*}
\end{proof}
\begin{proof}[Proof of Proposition \ref{p:gbddSG}]
Let $\omega'$ be as defined in Lemma \ref{l:extend}. For brevity set $\eps = \delta-\beta$. First of all observe that
\begin{align*}
\lim_{s\downarrow 0} s^{\eps}\n  [S(s)-\tilde{S}_0(s)] \n_{\calL(X_{\beta}^{A},X)} =0.
\end{align*}
Fix $\theta'\in (0,\theta)$. It follows from \cite[Theorem 1.7.7]{Pazy:83}, that one has, for all $t>0$:
\begin{align*}
S(t) &= \tinv{2\pi i}\int_{\omega'+\Gamma_{\theta'}} e^{\lambda t}R(\lambda:A) d\lambda;
\end{align*}
where $\Gamma_{\theta'}$ is the path composed from the two rays $r e^{i(\frac{\pi}{2}+\theta')}$ and $r
e^{-i(\frac{\pi}{2}+\theta')}$, $0\leq r <\infty$, and is oriented such that $\Im m(\lambda)$ increases along
$\Gamma_{\theta'}$. As $\omega'\geq \omega$, the integral is well-defined as $\calL(X)$-valued Bochner integral, and for $t>0$ one has:
\begin{align*}
\tfrac{d}{dt}S(t) &= \tinv{2\pi i}\int_{\omega'+\Gamma_{\theta'}} \lambda e^{\lambda t}R(\lambda:A) d\lambda;
\end{align*}
the integral again being well-defined as $\calL(X)$-valued Bochner integrals (see also the proof of \cite[Theorem
2.5.2]{Pazy:83}). Analogous identities hold for $\tilde{S}_0$ and $R(\lambda:\tilde{A}_0)$.\par
First let us assume that $\eps\in (0,1)$. Below we shall apply Lemma \ref{l:extend}, observing that for $r\in [0,\infty)$ we have
$$ |\omega'+re^{\pm i (\frac{\pi}{2}+\theta')}-\omega | \geq K_{\theta} r,$$ where $K_{\theta}$ is a constant depending only on $\theta$. Note that we use the coordinate transform $\lambda=\omega'+re^{\pm i (\frac{\pi}{2}+\theta')}$. For $s>0$ we have:
\begin{align*}
&\n S(s)-\tilde{S}_0(s) \n_{\calL(X_{\beta}^{A},X)} =\Big\n \tinv{2\pi i}\int_{\Gamma_{\theta'}} e^{\lambda
s}[R(\lambda:A)-R(\lambda:\tilde{A}_0)] d\lambda \Big\n_{\calL(X_{\beta}^{A},X)}\\
&\qquad\qquad   \leq \tinv{2\pi}  \int_{0}^{\infty} \big|e^{-i(\frac{\pi}{2}+\theta')+(\omega'+r e^{-i(\frac{\pi}{2}+\theta')})s
}\big|\\
&\qquad\qquad \qquad \times \big\n R(\omega'+re^{-i(\frac{\pi}{2}+\theta')}:A)-R(\omega'+re^{-i(\frac{\pi}{2}+\theta')}:\tilde{A}_0)\big\n_{\calL (X_{\beta}^{A},X)} \,dr \\
&\qquad\qquad  \quad + \tinv{2\pi}  \int_{0}^{\infty} \big| e^{i(\frac{\pi}{2}+\theta')+(\omega'+re^{i(\frac{\pi}{2}+\theta')})s
}\big|\\
&\qquad\qquad \qquad \times \big\n R(\omega'+re^{i(\frac{\pi}{2}+\theta')}:A)-R(\omega'+re^{i(\frac{\pi}{2}+\theta')}:\tilde{A}_0)\big\n_{\calL (X_{\beta}^{A},X)} \,dr
\\
&\qquad\qquad  \leq \tinv{\pi} C_{\omega,\theta,K,P_0}K_{\theta} D_{\delta}(A,A_0)e^{\omega's} \int_{0}^{\infty} r^{\eps-1}e^{-rs \sin\theta' }    \,dr \\
&\qquad\qquad   =\tinv{\pi} C_{\omega,\theta,K,P_0}K_{\theta} D_{\delta}(A,A_0)  [s\sin\theta']^{-\eps}e^{\omega's}\int_{0}^{\infty} u^{\eps-1}
e^{-u }
du  \\
&\qquad\qquad   = \tfrac{\Gamma(\eps)}{\pi} [\sin\theta']^{-\eps} C_{\omega,\theta,K,P_0}D_{\delta}(A,A_0) s^{-\eps}e^{\omega's}.
\end{align*}\par
For $\eps = 0$ one may avoid the singularity in $0$ in the usual way: for $s>0$ given we integrate over
$$\omega'+\Gamma_{\theta',s}= (\omega'+\Gamma^{(1)}_{\theta',s})\cup (\omega'+\Gamma^{(2)}_{\theta',s})\cup(\omega'+\Gamma^{(3)}_{\theta',s}),$$ where
$\Gamma^{(1)}_{\theta',s}$ and $\Gamma^{(2)}_{\theta',s}$ are the rays $r e^{i(\frac{\pi}{2}+\theta')}$ and $r
e^{-i(\frac{\pi}{2}+\theta')}$, $s^{-1} \leq r <\infty$, and $\Gamma^{(3)}_{\theta',s}= s^{-1}e^{i\phi}$,
$\phi\in[-\frac{\pi}{2}-\theta',\frac{\pi}{2}+\theta']$. This leads to the following estimate:
\begin{align*}
&\n S(s)-\tilde{S}_0(s) \n_{\calL(X_{\beta}^{A},X)} =\Big\n \tinv{2\pi i}\int_{\omega'+\Gamma_{\theta',s}} e^{\lambda
s}[R(\lambda:A)-R(\lambda:\tilde{A}_0)] d\lambda \Big\n_{\calL(X_{\beta}^{A},X)}\\
&\qquad \qquad \leq C_{\omega,\theta,K,P_0}K_{\theta}D_{\delta}(A,A_0) e^{\omega's}\Big[\tinv{\pi} \int_{s^{-1}}^{\infty} r^{-1}e^{-rs \sin\theta' }  
\,dr +
e \Big]\\
&\qquad \qquad \leq C_{\omega,\theta,K,P_0}K_{\theta}D_{\delta}(A,A_0) \big[[\pi \sin\theta']^{-1} e^{-\sin\theta'} + e\big]e^{\omega's} \\
& \qquad \qquad \leq 2
[\sin\theta']^{-1}C_{\omega,\theta,K,P_0}K_{\theta}D_{\delta}(A,A_0)e^{\omega's}.
\end{align*}
Recalling that $\eps = \delta-\beta$ this proves the uniform boundedness estimate of \eqref{p:gbddsg:unif1}.\par
Similarly to the above, for $\eps\in [0,1]$ and $s>0$ we have:
\begin{align*}
& \n \tfrac{d}{ds}S(s)-\tfrac{d}{ds}\tilde{S}_0(s) \n_{\calL (X_{\beta}^{A},X)}  =\Big\n \tinv{2\pi
i}\int_{\Gamma_{\omega'+\theta'}} \lambda e^{\lambda s}[R(\lambda:A)-R(\lambda:\tilde{A}_0)] d\lambda \Big\n_{\calL
(X_{\beta}^{A},X)}\\
&\qquad  \leq   \tinv{2\pi}e^{\omega's} \int_{0}^{\infty} r e^{- rs \sin\theta' }\big\n
 R(re^{-i(\frac{\pi}{2}+\theta')}:A)-R(re^{-i(\frac{\pi}{2}+\theta')}:\tilde{A}_0)\big\n_{\calL (X_{\beta}^{A},X)}
\,dr \\
&\qquad  \quad +\tinv{2\pi} e^{\omega's}\int_{0}^{\infty} re^{- rs \sin\theta' }\big\n
R(re^{i(\frac{\pi}{2}+\theta')}:A)-R(re^{i(\frac{\pi}{2}+\theta')}:\tilde{A}_0) \big\n_{\calL (X_{\beta}^{A},X)}
\,dr\\
& \qquad =\tinv{\pi} C_{\omega,\theta,K,P_0}K_{\theta}D_{\delta}(A,A_0)  [s\sin\theta']^{-1-\eps} e^{\omega's}\int_{0}^{\infty} u^{\eps}
e^{-u }
 du \\
&\qquad  = \tfrac{\eps\Gamma(\eps)}{\pi} [\sin\theta']^{-1-\eps} C_{\omega,\theta,K,P_0}K_{\theta}D_{\delta}(A,A_0) s^{-1-\eps} e^{\omega's} .
\end{align*}
Recalling that $\eps = \delta-\beta$ this proves the uniform boundedness estimate of \eqref{p:gbddsg:unif2}.\par

Concerning the $\gamma$-boundedness estimates, fix $\alpha >\eps$. By Proposition \ref{prop:integr-der} one has:
\begin{equation*}
\begin{aligned}
\gamma_{[X_{\beta}^{A},X]}\left(\{ s^{\alpha} [S(s)-\tilde{S}_0(s)]:s\in [0,t]\}\right) &  \leq \int_{0}^{t} \big\n
\tfrac{d}{ds} \big(s^{\alpha} [S(s)-\tilde{S}_0(s)]\big) \big\n_{\calL (X_{\beta}^{A},X)} \,ds  \\
& \leq \int_{0}^{t}  \alpha s^{\alpha-1}  \big\n S(s)-\tilde{S}_0(s) \big\n_{\calL (X_{\beta}^{A},X)}\,ds \\
&  \quad  + \int_{0}^{t} s^{\alpha}\big\n \tfrac{d}{ds}S(s)-\tfrac{d}{ds}\tilde{S}_0(s) \big\n_{\calL (X_{\beta}^{A},X)} \,ds.
\end{aligned}
\end{equation*}\par
Substituting \eqref{p:gbddsg:unif1} and \eqref{p:gbddsg:unif2} into the above one obtains that there exists a constant $C$ depending only on $\omega$, $\theta$, $K$, $\eps=\delta-\beta$, and $\n P_0\n_{\calL(X,X_0)}$ such that:
\begin{align*}
\gamma_{[X_{\beta}^{A},X]}\left(\{ s^{\alpha} [S(s)-\tilde{S}_0(s)]:s\in [0,t]\}\right)
& \leq C D_{\delta}(A,A_0) \int_{0}^{t}  e^{\omega' s} s^{\alpha-1-\eps}e^{\bar{\omega t}}
 \,ds  \\
& \leq C e^{(\omega' T)\maxsym 0} t^{\alpha-\eps} D_{\delta}(A,A_0),
\end{align*}
as $\alpha>\eps$.
\end{proof}
\begin{corollary}\label{cor:Holder}
Let the setting be as in Theorem \ref{t:app}. Let $\lambda \in [0,\inv{2})$ satisfy
\begin{align*}
0\leq \lambda < \min\{1-(\delta-\theta_F)\maxsym 0, \tinv{2}-\tinv{p}-(\delta-\theta_G)\maxsym 0\}.
\end{align*}
Suppose $x_0\in L^p(\Omega,\calF_0;X_{\delta}^{A})$ and $y_0 \in L^p(\Omega,\calF_0;X)$, then:
\begin{align*}
& \n U - Sx_0 - i_{X_0}(U^{(0)} - S_0P_0y_0) \n_{L^p(\Omega;C^{\lambda}([0,T];X))}\\
& \qquad \lesssim \n x_0-y_0\n_{L^p(\Omega,X)} + D_{\delta}(A,A_0)(1+\n x_0
\n_{L^p(\Omega;X_{\delta}^{A})}),
\end{align*}
with implied constant depending on $X_0$ only in terms of $\n P_0\n_{\calL(X,X_0)}$, on $A$ and $A_0$ only in terms of $1+D_{\delta}(A,A_0)$, $\omega$, $\theta$ and $K$, and on $F$ and
$G$ only in terms of their Lipschitz and linear growth constants Lip$(F)$, $\textrm{Lip}_{\gamma}(G)$, $M(F)$, and $M_{\gamma}(G)$.
\end{corollary}
\begin{proof}
As before, we write:
\begin{equation}\label{pert.holder}
\begin{aligned}
& \n U - Sx_0 - i_{X_0}(U^{(0)} - S_0P_0y_0) \n_{L^p(\Omega;C^{\lambda}([0,T];X))} \\
& \qquad = \Big\n t\mapsto \int_0^{t}\tilde{S}_0(t-s)[F(s,U(s))-F(s,\tilde{U}^{(0)}(s))]\,ds
\Big\n_{L^p(\Omega;C^{\lambda}([0,T];X))}\\
& \quad\qquad + \Big\n t\mapsto \int_0^{t}[S(t-s)-\tilde{S}_0(t-s)]F(s,U(s))\,ds
\Big\n_{L^p(\Omega;C^{\lambda}([0,T];X))}\\
& \quad\qquad + \Big\n t\mapsto \int_0^{t}\tilde{S}_0(t-s)[G(s,U(s))-G(s,\tilde{U}^{(0)}(s))]\,dW_H(s)
\Big\n_{L^p(\Omega;C^{\lambda}([0,T];X))}\\
& \quad\qquad + \Big\n t\mapsto \int_0^{t}[S(t-s)-\tilde{S}_0(t-s)]G(s,U(s))\,dW_H(s)
\Big\n_{L^p(\Omega;C^{\lambda}([0,T];X))}.
\end{aligned}
\end{equation}\par
For the first and second term on the right-hand side of \eqref{pert.holder} we apply Proposition \ref{prop:detHolder}. Note that as before we may pick $\eta_F,\eta_G\leq 0$ such that $\eta_F<\theta_F-\delta$ and $\eta_G<\theta_G-\delta$ and  
$$ \lambda <\min\{1+\eta_F,\tinv{2}-\tinv{p}+\eta_G\}.$$ Our choice of $Y_1$, $Y_2$, $\Phi$, $\Psi$ is the same as in part 2c, respectively 2d, of the proof of Theorem \ref{t:app}, whereas we set $\theta=1-\lambda$. This leads to the following estimates:
\begin{align*}
& \Big\n t\mapsto \int_0^{t}\tilde{S}_0(t-s)[F(s,U(s))-F(s,\tilde{U}^{(0)}(s))]\,ds
\Big\n_{L^p(\Omega;C^{\lambda}([0,T];X))}\\
& \qquad  \qquad\lesssim \n U - \tilde{U}^{(0)}(s) \n_{L^{p}(\Omega;L^{\infty}(0,T;X))}\leq \n U - \tilde{U}^{(0)}(s) \n_{\Vpc{\alpha}([0,T]\times \Omega;X)},
\end{align*}
and
\begin{align*}
&\Big\n t\mapsto \int_0^{t}[S(t-s)-\tilde{S}_0(t-s)]F(s,U(s))\,ds
\Big\n_{L^p(\Omega;C^{\lambda}([0,T];X))}\\
& \qquad \qquad \lesssim D_{\delta}(A,A_0)\n U \n_{L^{p}(\Omega;L^{\infty}(0,T;X))} \lesssim D_{\delta}(A,A_0)\big(1 + \n x_0\n_{L^p(\Omega,X)}\big).
\end{align*}\par
For the third and fourth term on the right-hand side of \eqref{pert.holder} we apply Corollary \eqref{cor:stochConv} with $\beta=\lambda$ and $\alpha\in (0,\inv{2})$ such that $\alpha>\lambda +\inv{p}+ \eta_G$. The choice of $Y_1$, $Y_2$, $\Phi$ and $\Psi$ is as in parts 2e and 2f of the proof of Theorem \ref{t:app}. This leads to:
\begin{align*}
& \Big\n t\mapsto \int_0^{t}\tilde{S}_0(t-s)[G(s,U(s))-G(s,\tilde{U}^{(0)}(s))]\,dW_H(s)
\Big\n_{L^p(\Omega;C^{\lambda}([0,T];X))}\\
& \qquad \qquad \lesssim \n U - \tilde{U}^{(0)}(s) \n_{\Vpc{\alpha}([0,T]\times \Omega;X)},
\end{align*}
and
\begin{align*}
& \Big\n t\mapsto \int_0^{t}[S(t-s)-\tilde{S}_0(t-s)]G(s,U(s))\,dW_H(s)
\Big\n_{L^p(\Omega;C^{\lambda}([0,T];X))}\\
& \qquad \qquad \lesssim  D_{\delta}(A,A_0) \n U \n_{\Vpc{\alpha}([0,T]\times \Omega;X)} \lesssim D_{\delta}(A,A_0)\big(1 + \n x_0\n_{L^p(\Omega,X)}\big).
\end{align*}\par
Combining these estimates with Theorem \ref{t:app} gives the desired result. It goes without saying that all the implied constants above depend on $X_0$ only in terms of $\n P_0\n_{\calL(X,X_0)}$, on $A$ and $A_0$ only in terms of $1+D_{\delta}(A,A_0)$, $\omega$, $\theta$ and $K$, and on $F$ and $G$ only in terms of their Lipschitz and linear growth constants Lip$(F)$, $\textrm{Lip}_{\gamma}(G)$, $M(F)$, and $M_{\gamma}(G)$.
\end{proof}

\section{Yosida approximations}\label{sec-yosida}
Consider \eqref{SDE} under the assumptions \MA, \MF, and \MG\ with the additional assumption that $\theta_F,\theta_G\geq 0$. We define $A_n:= nA R(n:A)$ to be the $n^{\textrm{th}}$
Yosida approximation of $A$. Let $U$ to denote the solution to \eqref{SDE} with operator $A$ and initial data $x_0\in L^p(\Omega,\calF_0;X)$
and, for $n\in \N$, let $U^{(n)}$ to denote the solution to \eqref{SDE} with operator $A_n$ instead of $A$ and initial data
$y_{0}\in L^p(\Omega,\calF_0;X)$.\par
\begin{theorem}\label{t:yosida}
For any $\eta\in [0,1]$ and $p\in (2,\infty)$ such that
$$\eta<\min\{\tfrac{3}{2}-\tinv{\tau}+\theta_F,\tfrac{1}{2}-\tinv{p}+\theta_G \}$$ and any $\alpha\in
[0,\inv{2})$ we have, assuming $y_0\in L^p(\Omega,\mathcal{F}_0;X_{\eta}^A)$:
\begin{align*}
\n U - U^{(n)}\n_{\Vpc{\alpha}([0,T]\times \Omega;X)}
&\lesssim  \n x_0 - y_0 \n_{L^p(\Omega;X)} + n^{-\eta}(1+\n y_0
\n_{L^p(\Omega;X_{\eta}^{A})}),
\end{align*}
with implied constants independent of $n$, $x_0$ and $y_0$.
\end{theorem}\par
The following corollary is a direct consequence of the Borel-Cantelli lemma and the above theorem (see Corollary \cite[Lemma 2.1]{Kloeneuen}):
\begin{corollary}\label{c:yosida_pathwise}
Let $\eta>0$ and $p\in (2,\infty)$ be such that $$\eta+\tinv{p}<\min\{\tfrac{3}{2}-\tinv{\tau}+\theta_F,\tfrac{1}{2}-\tinv{p}+\theta_G,1 \}$$ and assume $y_0=x_0\in L^{p}(\Omega,\mathcal{F}_0;X_{\eta}^A)$. Then there exists a random variable $\chi\in L^0(\Omega)$ such that for all $n\in \N$:
\begin{align*}
\n U - U^{(n)}\n_{C([0,T];\calH)} &\leq \chi n^{-\eta}.
\end{align*}
\end{corollary}\par
To prove Theorem \ref{t:yosida} we shall need the following lemma:
\begin{lemma}\label{l:fracpowYosida}
Let $\beta \in [0,1]$. Then there exists a constant $K'$ such that for all $n\geq 2\omega$ and all
$x\in X^{A}_{\beta}$ one has:
\begin{align*}
\n (2\omega I -A_n)^{\beta} x \n & \leq K'\n (2\omega I -A)^{\beta} x \n.
\end{align*}
\end{lemma}

\begin{proof}
Observe that
\begin{equation}
\begin{aligned}\label{omega-An}
2\omega I -A_n  = [(n+2\omega)I-4\omega^2 R(2\omega:A)](2\omega I - A)R(n:A).
\end{aligned}
\end{equation}
Thus for $x\in D(A)$ and $n\geq 2\omega$ we have:
\begin{align*}
\n (2\omega I-A_n)x \n &\leq \n[(n+2\omega)I-4\omega^2 R(2\omega:A)] R(n:A) \n_{\calL(X)} \n (2\omega
I -A)x \n\\
& \leq [K\tfrac{n+2\omega}{n-\omega}+K^2\tfrac{4\omega}{n-\omega}]\n (2\omega I -A)x \n\leq 4K(1+K)\n (2\omega I -A)x
\n.
\end{align*}
This proves the lemma for $\beta=1$. For $\beta=0$ the lemma is trivial. For $\beta\in (0,1)$ we need two extra
observations.\par
First of all, for $s>\omega$ and $\beta\in (0,1)$ we have, by definition, (see \cite[Section 2.6]{Pazy:83}):
\begin{align*}
(s I-A)^{-\beta} x & =\tfrac{\sin (\pi\beta)}{\pi} \int_{0}^{\infty} t^{-\beta} ((t+s)I - A)^{-1} x dt,
\end{align*}
and hence
\begin{equation}\label{fracPowEst1}
\begin{aligned}
\n (s I-A)^{-\beta}\n_{\calL(X)} & \leq K\tfrac{\sin (\pi\beta)}{\pi} \int_{0}^{\infty} t^{-\beta} (t+s-\omega)^{-1}  dt
\\
& \leq K \tfrac{\sin (\pi\beta)}{\pi}  \Big[ (s-\omega)^{-1} \int_{0}^{s-\omega} t^{-\beta} dt +
\int_{s-\omega}^{\infty}
t^{-1-\beta} dt \Big]\\
& = K\tfrac{\sin (\pi \beta)}{\pi \beta (1-\beta)} (s-\omega)^{-\beta}.
\end{aligned}
\end{equation}\par
Secondly, let $\mu,\lambda\in \omega+\Sigma_{\frac{\pi}{2}+\theta}$. We have: $$\n e^{-(\lambda I -A)R(\mu:A)t}
\n_{\calL(X)}=e^{-t}\n e^{(\mu-\lambda) R(\mu,A)t}\n_{\calL(X)}\leq e^{-t+\frac{|\mu-\lambda|}{|\mu-\omega|}Kt}.$$
Now suppose $n\geq 2\omega(1+4K)$, $\lambda=2\omega$, $\mu=\frac{n\lambda}{\lambda+n}=\frac{2\omega n}{2\omega+n}$. In
that case one may check that
$ \tfrac{|\mu-\lambda|}{|\mu-\omega|}K \leq \tinv{2},$
and thus that for $\beta\in (0,1)$:
\begin{equation*}
\begin{aligned}
\n [-(2\omega I -A)R(\tfrac{2\omega n}{2\omega+n}:A)]^{-\beta} \n_{\calL(X)} & = \Big\n
\tinv{\Gamma(\beta)}\int_{0}^{\infty}
t^{\beta-1}e^{-(2\omega I -A)R(\frac{2\omega n}{2\omega+n}:A)t}dt \Big\n_{\calL(X)}\\
& \leq 2^{\beta}.
\end{aligned}
\end{equation*}
As $[-(2\omega I -A)R(\tfrac{2\omega n}{2\omega+n}:A)]^{-\beta} \in \calL(X)$ for any $n\geq 2\omega$, it follows that
there exists a constant $M>0$ such that for all $n\geq 2\omega$:
\begin{align}\label{fracPowEst2}
\n [-(2\omega I -A)R(\tfrac{2\omega n}{2\omega+n}:A)]^{-\beta} \n_{\calL(X)} & \leq M.
\end{align}
\par
For $\beta\in (0,1)$ and $x\in X^A_{\beta}$ we have, by standard theory on functional calculus (see \cite{Haase}),
equation \ref{omega-An}, and the estimates \eqref{fracPowEst1} and \eqref{fracPowEst2}:
\begin{align*}
\n (2\omega I -A_n)^{\beta} x \n&= \n (n+2\omega)^{\beta}(\tfrac{2\omega n}{2\omega+n}I-A)^{\beta}(n I -A)^{-\beta}x \n
\\
& \leq (n+2\omega)^{\beta} \n \big[-(2\omega I-A)R(\tfrac{2\omega n}{2\omega+n}:A)\big]^{-\beta}\n_{\calL(X)} \\
& \quad \times \n (2\omega I -A)^{\beta}x \n\n (n I-A)^{-\beta}\n_{\calL(X)}\\
& \leq 4\tfrac{\sin (\pi \beta)}{\pi \beta (1-\beta)} KM \n (2\omega I -A)^{\beta}x \n.
\end{align*}
\end{proof}
% %
\begin{proof}[Proof of Theorem \ref{t:yosida}.]
Without loss of generality we may assume $\omega\geq 0$. In order to apply Theorem \ref{t:app}, we must prove that
$A_n$, $n\geq 2\omega$, are of uniform type, i.e., that there exist $\bar{\omega}\in \R, \bar{\theta}\in
(0,\frac{\pi}{2})$ and $\bar{K}>0$ such that $A_n$ is of type $(\bar{\omega}, \bar{\theta},\bar{K})$ for all $n\geq
2\omega$. Fix $n\geq 2\omega$. One checks that:
\begin{align}\label{yosres}
R(\lambda:A_n) & = (n+\lambda)^{-1}(n-A)R(\tfrac{\lambda n}{n+\lambda}:A)
\end{align}
whenever $\frac{\lambda n}{n+\lambda} \in \omega + \Sigma_{\frac{\pi}{2}+\theta}$. Define $f:\C\rightarrow \C$;
$f(z)=\frac{nz}{n-z}$. From \eqref{yosres} it follows that $\lambda\in \rho(A)$ if and only if $f(\lambda)\in
\rho(A_n)$. By
standard theory on M\"obius transforms we have that
$$ f\big(\{\omega+\Sigma_{\frac{\pi}{2}+\theta}\}\big) = \C \setminus \big( D_{1}\cap D_{2}\big), $$
where $D_1$ and $D_2$ are both closed disks with radius $\frac{n^2}{2(n-\omega)\cos \theta}$; the center of $D_1$ is in
$\frac{n}{2(n-\omega)}(2\omega-n, \tan(\theta))$ and the center of $D_2$ is is in
$\frac{n}{2(n-\omega)}(2\omega-n,-\tan(\theta))$. The boundaries of these disks intersect each other on the real axis at
the points $-n$ and $\frac{n\omega}{n-\omega}$. The angle at intersection is $\pi-2\theta$. As $n\geq 2\omega$ we have
$\frac{n\omega}{n-\omega}\leq 2\omega$ and thus $\rho(A_n)\subset 2\omega+\Sigma_{\frac{\pi}{2}+\theta}$. It remains to
prove the desired estimate on the resolvent.\par
Using \eqref{yosres} one may check that for $\lambda \in 2\omega+\Sigma_{\frac{\pi}{2}+\theta}$ we have:
\begin{align}\label{YosidaResComp}
R(\lambda:A)-R(\lambda:A_n) & = -(\lambda+n)^{-1}A^2R(\tfrac{\lambda n}{n+\lambda}:A)R(\lambda:A).
\end{align}
Thus by \eqref{AResEst} we have, for $\lambda \in  \omega(1+2(\cos\theta)^{-1}) +\Sigma_{\frac{\pi}{2}+\theta}$:
\begin{align*}
\n R(\lambda:A)-R(\lambda:A_n) \n_{\calL(X)} & \leq (1+2K)^2|\lambda+n|^{-1} \leq (1+2K)^2|\lambda-\omega|^{-1}.
\end{align*}
The final estimate follows from the fact that by standard theory on M\"obius transforms we have that
$\tfrac{|\lambda-\omega|}{|\lambda+n|}\leq 1$ for $\lambda \in \omega+\Sigma_{\frac{\pi}{2}+\theta}$. In conclusion we
have, for $\lambda \in  \omega(1+2(\cos\theta)^{-1}) +\Sigma_{\frac{\pi}{2}+\theta}$:
\begin{align*}
\n R(\lambda:A_n) \n_{\calL(X)} & \leq \n R(\lambda:A)\n_{\calL(X)}+ \n R(\lambda:A)-R(\lambda:A_n) \n_{\calL(X)}  \\
& \leq [K+(1+2K)^2]|\lambda-\omega|^{-1}.
\end{align*}
This proves that $A_n$ is of type $(\omega(1+2(\cos\theta)^{-1}),\theta,K+(1+2K)^2)$ for all $n\geq 2\omega$.\par
It also follows from \eqref{YosidaResComp} that if we take, for example, $\lambda_0= \omega(1+2(\cos\theta)^{-1})$, then
we have, for $n\geq 2\omega$:
\begin{align*}
\n R(\lambda_0:A)-R(\lambda_0:A_n)\n_{\calL(X)} & \leq (1+2K)^{2}n^{-1}.
\end{align*}
In other words, for all $n\in \N$ condition \eqref{app:ass} in Theorem \ref{t:app} is satisfied with $\delta=1$ and
$\lambda_0= \omega(1+2(\cos\theta)^{-1})$. In particular we can apply Theorem \ref{t:app} to obtain the desired result
for the case $\theta_F>-\frac{1}{2}+\inv{\tau}$, where $\tau$ is the type of $X$, and $\theta_G> \inv{2}+\inv{p}$. Concerning the dependence on $1+D(A,A_n)$ of the implied constant in \eqref{pert_est}, note that $1+D(A,A_n)$ is uniformly bounded in $n$, both from above and away from $0$.\par
In order to get the desired result for general $\theta_F,\theta_G\geq 0$ we consider the difference $R(\lambda_0:A)-R(\lambda_0:A_n)$ in the $\calL(X_{\delta-1}^{A_n},X)$-norm.
(Note that if $A$ is unbounded then $R(\lambda_0:A)-R(\lambda_0:A_n)\notin \calL(X^{A}_{\delta-1},X)$ for any $\delta<1$.) For $n\geq
\omega(1+2(\cos\theta)^{-1})$ we have, by \eqref{AResEst}, that $\n (2\omega I-A_n) \n_{\calL(X)}\leq
2n(1+K)$. Thus by Theorem \ref{t:fracPowInt} we have, for $\delta \in (0,1)$:
\begin{align*}
\n (2\omega I - A_n)^{1-\delta} x \n & \leq 2(1+2K) \n x \n^{\delta} \n A_n x\n^{1-\delta} \\
& \leq 2^{2-\delta}(1+2K)^{2-\delta} n^{1-\delta} \n x \n.
\end{align*}
It follows that for $\delta \in [0,1)$ we have:
\begin{align*}
\n R(\lambda_0:A)-R(\lambda_0:A_n) \n_{\calL(X^{A_n}_{\delta-1},X)} & \leq 2^{2-\delta}(1+2K)^{4-\delta} n^{-\delta}.
\end{align*}\par
We are now ready to apply Theorem \ref{t:app}. First of all observe that by Lemma \ref{l:fracpowYosida} we have that
$F:[0,T]\times X\rightarrow X_{\theta_F}^{A_n}$ is Lipschitz continuous and of linear growth for all $n\geq 2\omega$
with Lipschitz and growth constants independent
of $n$, and $G:[0,T]\times X\rightarrow \gamma(H,X_{\theta_G}^{A_n})$ is $L^2_{\gamma}$-Lipschitz continuous and of
linear growth for all $n\geq
2\omega$ with Lipschitz and growth constants independent of $n$. Also, $1+D(A,A_n)$ is uniformly bounded in $n$.\par
Fix $\eta\in [0,1]$ such that $\eta<\min\{\frac{3}{2}-\inv{\tau}+\theta_F,\frac{1}{2}-\inv{p}+\theta_G \}$ and suppose $y_0\in
L^p(\Omega,\mathcal{F}_0;X_{\eta}^{A})$. It follows from Theorem \ref{t:app} with $\delta=\eta$, but with $A_n$ playing
the role of $A$ and $A$ playing the role of $A_0$, that:
\begin{align*}
\n U - U^{(n)}\n_{\Vpc{\alpha}([0,T]\times \Omega;X)} &\lesssim  \n x_0 - y_0\n_{L^p(\Omega;X)}
+ n^{-\eta}(1+\n y_0 \n_{L^p(\Omega;X_{\eta}^{A})}),
\end{align*}
with implied constants independent of $n$, $x_0$ and $y_0$.
\end{proof}

\end{document}